\documentclass[reqno]{amsart}
\usepackage{mathrsfs}
\usepackage{amscd,amsfonts,amsmath,amsthm,amssymb,stmaryrd}
\usepackage{graphicx, color}
\usepackage{multirow}
\allowdisplaybreaks[4]
\providecommand{\norm}[1]{\left\Vert#1\right\Vert}

\def\XXint#1#2#3{{\setbox0=\hbox{$#1{#2#3}{\int}$ }
\vcenter{\hbox{$#2#3$ }}\kern-.6\wd0}}

\newtheorem{Theorem}{Theorem}[section]
\newtheorem{Lemma}[Theorem]{Lemma}
\newtheorem{Proposition}{Proposition}[section]
\theoremstyle{definition}
\newtheorem{Definition}{Definition}[section]
\newtheorem{Corollary}{Corollary}[section]
\theoremstyle{remark}
\newtheorem{Remark}{Remark}[section]
\numberwithin{equation}{section} \allowdisplaybreaks

\author{Zhong Tan}
\address{ School of Mathematical Sciences and Fujian Provincial Key Laboratory on Mathematical Modeling and Scientific Computing, Xiamen University \\ Xiamen,
Fujian, 361005, P. R. China.} \email[Z. Tan]{ztan85@163.com}

\author{Huaqiao Wang}
\address{Institute of Applied Physics and Computational Mathematics,
P.O.Box 8009, Beijing 100088, China.} \email[H.Q.
Wang]{hqwang111@163.com}

\author{Yucong Wang}
\address{School of Mathematical Sciences, Xiamen University, Xiamen, 361005, China.}
\email[Y.C. Wang]{yucongwang666@163.com}
\title[Existence of martingale solutions to the stochastic Non-Newtonian Fluids]
{Existence of a martingale weak solution to the Equations of
Non-Stationary Motion of Non-Newtonian Fluids with a stochastic
perturbation} \keywords{$L^{\infty}$-truncation, martingale weak
solution, Non-Newtonian fluids, pressure
decompsition.}\subjclass[2010]{35R60, 35D30, 60H15, 35K55, 76D03.}
\date{\today}

\begin{document}
\begin{abstract}
In this paper, we consider the stochastic 
incompressible non-Newtonian fluids driven by a cylindrical Wiener
process $W$ with shear rate dependent on viscosity in a bounded
Lipschitz domain $D\in \mathbb{R}^n$ during the time interval
$(0,T)$. 
For 
$q>\frac{2n+2}{n+2}$ in the growth conditions
\eqref{1.2}, we prove the existence of a martingale weak solution
with $\nabla\cdot u=0$ by using a pressure decomposition which is
adapted to the stochastic setting, the stochastic compactness method
and the $L^\infty$-truncation.
\end{abstract}

\maketitle
\section{Introduction}
\noindent Let $D\in \mathbb{R}^n$ ($n\geq 2$) be a bounded Lipschitz
domain. For the time interval $(0,T)$, we set $Q:=(0,T)\times D$. In
this paper, we consider the following equations:
\begin{equation}\label{1.1}
\begin{cases}
du+\nabla\cdot(u\otimes u-S+p\text{I})dt=fdt+\Phi(u)dW,\\
\nabla\cdot u=0,\\
u|_{\partial D}=0,\\
u|_{t=0}=u_0,
\end{cases}
\end{equation}
where $S=\{S_{ij}\}$ is the deviatoric stress tensor, $p$ the
pressure, $u$ the velocity, $f$ the external force and $W$ a
cylindrical Wiener process with values in a Hilbert space. $\Phi$
satisfy the linear growth assumption (see Sect.2 for details).

The stress $S$ may depend on both $(x,t)$ and the ``rate of strain
tensor" $\mathbb{D}=\{\mathbb{D}_{ij}\}$, which is defined by
$\mathbb{D}_{ij}=\mathbb{D}_{ij}(u):=\frac{1}{2}(\partial_{x_j}u^i+\partial_{x_i}u^j)$,
$i,j=1,\cdots ,n$. We refer to \cite{GAGM}, \cite{GKB} and \cite{HL}
about the continuum mechanical background. As far as we know, the
fluids with shear dependent viscosity are often used in engineering
practice.  So it's meaningful to study this kind of fluid. In this
paper, $S$ is assumed to be a function of the shear rate and the
constitutive relations reads as
$$S=\nu(\mathbb{D}_{\text{II}})\mathbb{D},$$ where
$\mathbb{D}_{\text{II}}=\frac{1}{2}\mathbb{D}:\mathbb{D}$ is the
second invariant of $\mathbb{D}$. Here are some examples of precise
constructions of $S$: for $q\in (1,+\infty)$, constant $\nu_0$,
\begin{align*}
S&=\nu_0(\mathbb{D}_{\text{II}})^{\frac{q-2}{2}}\mathbb{D},\\
S&=\nu_0(1+\mathbb{D}_{\text{II}})^{\frac{q-2}{2}}\mathbb{D}.
\end{align*}
For detials, see \cite{GAGM, RBBRCAOH, WLW}. If $q\in(1,2)$, we say
the non-Newtonian fluids is pseudoplastic or shear thinning (for
example, ketchup); if $q=2$, it's Newtonian fluids; if
$q\in(2,+\infty)$, we say the non-Newtonian fluids is dilatant or
shear thickening (for example, batter). The following two
constitutive laws which are also of interest in engineering practice
are given by
\begin{align*}S=\nu_0(\mathbb{D}_{\text{II}})^{\frac{q-2}{2}}\mathbb{D}+\nu_{\infty}\mathbb{D},\quad\\
S=\nu_0(1+\mathbb{D}_{\text{II}})^{\frac{q-2}{2}}\mathbb{D}+\nu_{\infty}\mathbb{D},
\end{align*}
where $\nu_0$ and $\nu_{\infty}$ are positive constants and $q\in
[1,+\infty)$. An extensive list for specific $q$-values for
different fluids can be found in \cite{RBBRCAOH}.

 For $q\in [1,+\infty)$, we assume the
deviatoric stress tensor $S$ satisfy the following conditions in
this paper: $S:Q\times
\mathbb{M}_{\text{sym}}^n\rightarrow\mathbb{M}_{\text{sym}}^n$ is a
Carath\'{e}odory function. $\forall \xi\in
\mathbb{M}_{\text{sym}}^n$ (vector space of all symmetric $n\times
n$ matrices $\xi=\{\xi_{ij}\}$. We equip $\mathbb{M}_{\text{sym}}^n$
with scalar product $\xi:\eta$ and norm
$\norm{\xi}:=(\xi:\xi)^\frac{1}{2}$.), for almost
all $(x,t)\in Q$,
\begin{equation}\label{1.2}
|S(x,t,\xi)|\leq C_0\norm{\xi}^{q-1}+\eta_1,
\end{equation}
where $C_0>0$, $\eta_1\geq0$, $\eta_1\in L^{q'}(Q)$,
$1/q+1/{q^\prime}=1$; $\forall \xi\in \mathbb{M}_{\text{sym}}^n$,
for almost all $(x,t)\in Q$,
\begin{equation}\label{1.3}
\begin{split}
S(x,t,\xi):\xi\geq C_0\norm{\xi}^{q}-\eta_2,
\end{split}
\end{equation}
where $C_0>0$, $\eta_2\geq0$, $\eta_2\in L^{1}(Q)$; $\forall
\xi,\eta\in \mathbb{M}_{\text{sym}}^n(\xi\neq\eta)$, for almost all
$(x,t)\in Q$,
\begin{equation}\label{1.4}
(S(x,t,\xi)-S(x,t,\eta): (\xi-\eta)>0.
\end{equation}

The flow of a homogenous incompressible fluid without stochastic
part is described by the following equations:
\begin{equation}\label{1.5}
\begin{cases}
\partial_t u+\nabla\cdot(u\otimes u-S+p\text{I})=-\nabla\cdot f,\\
\nabla\cdot u=0,\\
u|_{\partial D}=0,\\
u|_{t=0}=u_0.
\end{cases}
\end{equation}

In the late sixties, Lions and Ladyshenskaya in
\cite{LOA1,LOA2,LOA3,LJL} started the mathematical discussion of
power-law model. In \cite{LOA1}, Ladyzhenskaya achieved the
existence and uniqueness of weak solutions and in \cite{LJL} Lions
achieved these results for $q\geq\frac{3n+2}{n+2}$. They showed the
existence of a weak solution in the space
$L^q(0,T;W^{1,q}_{0,\text{div}}(D))\cap L^{\infty}(0,T;L^2(D))$. In
this particular case that $u\otimes u:\mathbb{D}(u)\in L^1(Q) $
follows from parabolic interpolation, the proof of existence is
based on monotone operators and compactness arguments. In
\cite{JMJNMR}, M\'{a}lek, Ne\v{c}as and Ru\v{z}i\v{c}ka proved the
existence for $q\in [2,\frac{11}{5})$ under the assumption that
$D\in \mathbb{R}^3$ is a bounded domain with $C^3$-boundary and that
$S(\mathbb{D})$ has the form
$S(\mathbb{D})=\partial_\mathbb{D}\Phi(\mathbb{D}_{\textrm{II}}).$
In \cite{JW}, Wolf improved this result to the case
$p>\frac{2n+2}{n+2}$ by using $L^{\infty}$-truncation.

In the fluid motion, apart from the force $f$, there might be
further quantities with a influence on the motion. This influence
usually is small and can be shown by adding a stochastic part to the
equation. The stochastic part to the equation can be understood as a
turbulence. This type of equation is often used in fluid mechanics
since they model the phenomenon of perturbation.  So it's very
interesting to study the stochastic fluids. In SPDES, we consider
two concepts: strong (pathwise) solutions and
 weak (martingale) solutions. Strong solutions means that the
underlying probability space and the Wiener process are given in
advance. While martingale solutions means that the combination of
these stochastic elements and the fluid variables is the solution of
the problem and the original equations are satisfied in the sense of
distributions. Clearly, the existence of strong solutions implies
the existence of martingale solutions.  There are many research
results on the stochastic Newton flow dating back to the 1970's with
the initial work of Bensoussan and Temam \cite{ABRT}. For example,
 the existence
of strong solutions and martingale solutions to the stochastic
incompressible Navier-Stokes equations is established by Da
Prato-Zabczyk \cite{GDJZ}, Breckner \cite{HB}, Menaldi-Sritharan
\cite{JLMSSS}, Glatt-Holtz-Ziane \cite{NGMZ}, Taniguchi \cite{TT},
Cap\'{i}nski-Peszat \cite{MCSP,CP-2001}, Kim \cite{JUK},
Cap\'{i}nski-Gatarek \cite{CG-1994}, Flandoli-Gatarek \cite{FG},
Mikulevicius-Rozovskii \cite{MB-2004, MB-2005}, Brze\'{z}niak-Motyl
\cite{BM} and the references therein; for the stochastic
incompressible MHD equations, the existence of solutions is
considered in \cite{ZTDWHW} and the references therein. For the
stochastic incompressible non Newtonian flow, there are only a few
results. Recently, Breit \cite{BD} proved the existence of a
martingale weak solution of the stochastic Navier-Stokes equations
of the model $S(\mathbb{D}(u))=(1+\mathbb{D}u)^{p-2}\mathbb{D}u$. In
this paper, we will prove the existence of martingale solutions of
the stochastic equations \eqref{1.1} with
$S=\nu(\mathbb{D}_{\text{II}})\mathbb{D}$, which is the general form
of $S(\mathbb{D}(u))=(1+\mathbb{D}u)^{p-2}\mathbb{D}u$.

 Comparing
with the work in \cite{JW}, we face the essential challenge of
establishing sufficient compactness in order to be able to pass to
the limit in the class of solutions. In general it is not possible
to get any compactness in $\omega$ as no topological structure on
the sample space $\Omega$. That is, even if a space $\mathcal{X}$ is
compactly embedded in another space $\mathcal{Y}$, it is not usually
the case that $L^2(\Omega,\mathcal{X})$ is compactly embedded in
$L^2(\Omega,\mathcal{Y})$. As such, Aubin-Lions Lemma or
Arzel\`{a}-Ascoli Theorem, which classically make possible the
passage to the limit in the nonlinear terms, cannot be directly
applied in the stochastic setting. To overcome this difficulty, it
is classical to rather concentrate on compactness of the set of laws
of the approximations (the Prokhorov Theorem, which is used to
obtain compactness in the collection of probability measures
associated to the approximate solutions) and apply the Skorokhod
embedding Theorem, which provides almost sure convergences of a
sequence of random variables that have the same laws as the original
ones, but relative to a new underlying stochastic basis. However,
the Skorokhod embedding Theorem is restricted to metric spaces but
the structure of the stochastic non Newtonian equations naturally
leads to weakly converging sequences. For this, we apply the
Jakubowski-Skorokhod Theorem which is valid on a large class of
topological spaces (including separable Banach spaces with weak
topology).  Compared with the work in \cite{BD}, the biggest
difference is that we use the cut-off function to prove the
approximated equations for
$S=\nu(\mathbb{D}_{\text{II}})\mathbb{D}$, which is the general form
of $S$ in \cite{BD} also hold on the new probability space, rather
than using a general and elementary method that was recently
introduced in \cite{MO}.

The rest of the paper is organized as follows. In Sect 2, we
formulate some stochastic background and give our main Theorem. In
Sect 3, we reconstructed the pressure which disappears in the weak
formulation. In Sect 4, we use Galerkin method adding a large power
of $u$ to study auxiliary problem. In Sect 5, we prove the main
theorem.

\section{Hypotheses and Main Theorem}
Let $(\Omega,\mathscr{F},\mathscr{F}_{t},\mathbb{P})$ be a
stochastic basis, where $\mathscr{F}_t$ is a nondecreasing family of
sub-$\sigma$-fields of $\mathscr{F}$, i.e.,
$\mathscr{F}_s\subset\mathscr{F}_t$ for $0\leq s\leq t\leq T$.
Assume that filtration $\{\mathscr{F}_t,0\leq t\leq T\}$ is
right-continuous and $\mathscr{F}_0$ contains all the
$\mathbb{P}$-negligible events in $\mathscr{F}$.

The process $W$ is a cylindrical Wiener process, i.e.,
$W(t)=\sum_{k\geq1}\beta_k(t)e_k$, with $(\beta_k)_{k\geq1}$ being
mutually independent real-valued standard Wiener processes relative
to $\mathscr{F}_{t}$ and $\{e_k\}_{k\geq1}$ a complete orthonormal
system in a separable Hilbert space ${U}$. Since $W$ don't actually
converge on $U$, we define $U_0\supset U$ by
$U_0=\{v=\sum_{k\ge1}\alpha_ke_k;\sum_{k\ge 1}{{\alpha_k^2}/
k^2}<\infty \}$. The norm of $U_0$ is given by
$\|v\|^2_{U_0}=\sum_{k\ge 1}{\alpha_k^2}/k^2, v=\sum_{k\ge
1}\alpha_ke_k$. Then the embedding ${U}\hookrightarrow {U}_0$ is
Hilbert-Schmidt and the trajectories of $W$ are $\mathbb{P}$-a.s.
continuous with values in ${U}_0$. Note that
$$\int_0^t \psi(r)dW(r)$$
where $\psi \in L^2(\Omega; L_2({U},L^2(D)))$ is progressively
measurable, defines a $\mathbb{P}$-almost surely continuous
$L^2(\Omega)$ valued $\mathscr{F}_t$-martingale. Furthermore, we can
multiply the It\^{o}'s integral with a test-function since
$$\int_0^t\int_{D} \psi(r)\cdot\varphi dxdW(r)
=\sum_{k=1}^{\infty}\int_0^t\int_{D}\psi(r)e_k\cdot\varphi
dxd\beta_k(r),\,\varphi\in L^2(D)$$ is well-defined.

In this paper, the mapping $\Phi(z):{U}\rightarrow L^2(D)$ is
defined by $\Phi(z)e_k=g_k(z(\cdot))$, $\forall z\in L^2(D)$. We
assume that $g_k\in C(\mathbb{R}\times D)$ and satisfy the following
condition:
\begin{equation}\label{2.1}
\begin{split}
\sum_{k\geq1}|g_k(\xi)|\leq c(1+|\xi|),\, \forall \xi \in
\mathbb{R}^n,
\end{split}
\end{equation}
\begin{equation}\label{2.2}
\begin{split}
\sum_{k\geq1}|\nabla g_k(\xi)|^2\leq c,\, \forall \xi \in
\mathbb{R}^n,
\end{split}
\end{equation}
and additionally implies
\begin{equation}\label{2.3}
\begin{split}
\sup_{k\geq1}k^2|g_k(\xi)|^2\leq c(1+|\xi|^2),\, \forall \xi \in
\mathbb{R}^n.
\end{split}
\end{equation}

Now, we are ready to give a precise definition of the  martingale
weak solutions.
\begin{Definition}\label{def2.1}
Let  $\mu_0$, $\mu_f$ be Borel probability measures on
$L^2_{\text{div}}(D)$ and $L^2(Q)$ respectively. A system
$$((\Omega,\mathscr{F},\mathscr{F}_{t},\mathbb{P}),u,u_0,f,W)$$
is called a martingale weak solution to \eqref{1.1}, and $S$ satisfy
\eqref{1.2}, \eqref{1.3}, and \eqref{1.4} with the initial datum
$\mu_0$ and $\mu_f$  if the following conditions are satisfy:

(1) $(\Omega,\mathscr{F},\mathscr{F}_t,\mathbb{P})$ is a stochastic
basis with a complete right-continuous filtration,

(2) $W$ is an $\mathscr{F}_t$-cylindrical Wiener process,

(3) $u\in L^2(\Omega; L^{\infty}(0,T;L^2(D)))\cap L^q(\Omega;
L^{q}(0,T;W^{1,q}_{0,\text{div}}(D)))$ is progressively measurable,

(4) $u_0\in L^2(\Omega; L^{2}(D))$ with $\mu_0=\mathbb{P}\circ
u_0^{-1}$,

(5) $f\in L^2(\Omega; L^{2}({Q}))$ is adapted to
$\mathscr{F}_t$ and $\mu_f=\mathbb{P}\circ f^{-1}$,

(6) $\forall \varphi\in C_{0,\text{div}}^{\infty}(D)$ and $\forall
t\in [0,T]$, it holds that $\mathbb{P}$-a.s.
\begin{align*}
\int_{D}(u(t)-u_0)\cdot\varphi
dx\!&=\!\!\int_0^t\!\!\!\int_{D}u\otimes u
:\mathbb{D}(\varphi)-S(x,r,\mathbb{D}(u)):\mathbb{D}(\varphi) dxdr\\
&\quad+\!\int_0^t\!\!\!\int_{D}f\cdot \varphi dx dr
+\!\int_0^t\!\!\!\int_{D}\Phi(u)\cdot \varphi dxdW(r).
\end{align*}

\end{Definition}

Next, we state our main result.

\begin{Theorem}
Assume that $q>\frac{2n+2}{n+2}$, $S$ satisfies \eqref{1.2},
\eqref{1.3} and \eqref{1.4}. $\Phi$ satisfies \eqref{2.1} and
\eqref{2.2}. And further suppose that
\begin{equation}\label{2.4}
\begin{split}
\int_{L^2_{\rm{
div}}(D)}\|v\|^{\beta}_{L^2(D)}d\mu_0(v)<\infty,\quad
\int_{L^2(Q)}\|\mathbf{g}\|^{\beta}_{L^2(Q)}d\mu_f(\mathbf{g})<\infty
\end{split}
\end{equation}
with $\beta:=\max\{\frac{2n+2}{n},\frac{qn+2q}{n}\}$. Then there
exists a martingale weak solution to \eqref{1.1} in the sense of
Definition \ref{def2.1}.
\end{Theorem}

\section{Pressure Decomposition}
In the present section we are going to introduce a pressure method generalizes \cite{JW} to the stochastic case.
Here the pressure $p$ will be decomposed into four part $p_1$, $p_2$, $p_h$ and $p_\Phi$.
We show a-priori estimates for the components $p_1$, $p_2$, $p_h$ and $p_\Phi$.
\begin{Theorem}\label{theorem3.1}
Let $(\Omega, \mathscr{F}, \mathscr{F}_{t}, \mathbb{P})$ be a
stochastic basis, $v\in L^2(\Omega; L^\infty(0,T;L^2(D)))$ adapted
to $\mathscr{F}_t$. Assume $H_1+H_2\in L^\alpha(\Omega;
L^\alpha(Q))$ adapted to $\mathscr{F}_t$ for some $\alpha>1$,
$H_1\in L^{\alpha_1}(\Omega\times Q,\mathbb{P}\otimes
\mathcal{L}^{n+1})$ and $H_2,\nabla H_2\in L^{\alpha_2}(\Omega\times
Q,\mathbb{P}\otimes \mathcal{L}^{n+1})$. Moreover, let $v_0\in
L^2(\Omega; L^2_{\rm{div}}(D))$, and $\Phi\in L^2(\Omega;
L^\infty(0,T;L_2({U},L^2(D))))$ progressively measurable such that
\begin{equation}\label{3.1}
\begin{split}
\int_{D}(v(t)-v_0)\cdot\varphi
dx+\int_0^t\!\!\!\int_{D}(H_1+H_2):\nabla\varphi
dxdr=\int_0^t\!\!\!\int_{D}\Phi\cdot\varphi dx dW(r)
\end{split}
\end{equation}
holds for all $\varphi\in C_{0,\rm{div}}^\infty(D)$. Then there are
functions $p_1$, $p_2$, $p_h$ and $p_\Phi$ adapt to $\mathscr{F}_t$
such that

 $(1)$ $\Delta p_h=0$ and the following estimates are satisfied for
$\theta:=\min\{2,\alpha\}$:
\begin{align*} E\left(
\int_Q|p_1|^{\alpha_1}dxdt\right)^\beta&\leq cE\left(
\int_Q|H_1|^{\alpha_1}dxdt\right)^\beta,\\
E\left(\int_Q|p_2|^{\alpha_2}dxdt\right)^\beta&\leq cE\left(
\int_Q|H_2|^{\alpha_2}dxdt\right)^\beta,\\
E\left(\int_0^T\!\!\!\int_{D'}|\nabla
p_2|^{\alpha_2}dxdt\right)^\beta&\leq cE\left(
\int_Q|H_2|^{\alpha_2}+|\nabla H_2|^{\alpha_2}dxdt\right)^\beta,\\
E\left(\int_Q|p_1+p_2|^{\alpha}dxdt\right)&\leq cE\left(
\int_Q|H_1+H_2|^{\alpha}dxdt\right),\\
E\left(\sup_{t\in(0,T)} \int_{D}|p_\Phi|^{2}dx\right)&\leq cE\left(
\sup_{t\in(0,T)}\|\Phi\|^2_{L_2({U},L^2(D))}\right),
\end{align*}
\begin{align*}
E\left(\sup_{t\in(0,T)} \int_{D}|p_h|^{\theta}dx\right)&\leq
cE\bigg( 1+\sup_{t\in(0,T)} \int_{D}|v|^{2}dx
+\sup_{t\in(0,T)}\|\Phi\|^2_{L_2({U},L^2(D))}\\
&\quad+\int_{D}|v_0|^{2}dx+\int_Q|H_1+H_2|^{\alpha}dxdt\bigg),\\
E\left(\sup_{t\in(0,T)} \int_{D}|p_h|^{\theta}dx\right)^\beta&\leq
cE\left(\sup_{t\in(0,T)} \int_{D}|v|^{2}dx
+\sup_{t\in(0,T)}\|\Phi\|^2_{L_2({U},L^2(D))}\right)^\beta\\
&\quad+cE\left(1+\int_{D}|v_0|^{2}dx+\int_Q|H_1+H_2|^{\alpha}dxdt\right)^\beta,
\end{align*}
for all $1\leq\beta<\infty$ and $D'\subset\subset D$.

$(2)$ for all $\varphi\in C_0^\infty(D)$, it holds that
\begin{equation*}
\begin{split}
&\int_{D}(v(t)-v_0-\nabla p_h(t))\cdot\varphi dx+\int_0^t\!\!\!\int_{D}(H_1+H_2):\nabla\varphi dxdr\\
&=\int_0^t\!\!\!\int_{D}(p_1+p_2){\rm{div}\varphi}
dxdr+\int_{D}p_\Phi(t){\rm{div}\varphi}
dx+\int_0^t\!\!\!\int_{D}\Phi \cdot\varphi dxdW(r).
\end{split}
\end{equation*}
Moreover, we have $p(t)=p_h(t)+p_\Phi(t)+\int_0^t(p_1+p_2)dr\in
L^\theta(\Omega;L^\infty(0,T;L^\theta(D)))$ and
$p_1(0)=p_2(0)=p_h(0)=p_\Phi(0)=0$ $\mathbb{P}$-a.s..
\end{Theorem}

\begin{proof}
Let $v$ be a weak solution to \eqref{3.1} for all $\varphi\in
W_{0,\text{div}}^{1,\theta'}(D),1/\theta+1/\theta'=1$. Then by De
Rahm's theorem (see \cite{GPG}), there exists a unique function
$p(t)\in L_0^\theta(D)$ with $p(0)=0$, such that
\begin{align*}
\int_{D}(v(t)-v_0)\cdot\varphi dx-\int_{D}p(t){\rm{div}}\varphi
dx+\int_0^t\!\!\!\int_{D}(H_1+H_2):\nabla\varphi
dxdr=\int_0^t\!\!\!\int_{D}\Phi\cdot\varphi dx dW(r),
\end{align*} for all
$\varphi\in W_{0}^{1,\theta'}(D)$.

By using the Bogovski\u{i}-operator $\text{Bog}_D$ (see \cite{MEB})
and let $\mathcal{B}=Bog_{D}(\varphi-(\varphi)_D)$ where
$(\varphi)_D=\frac{1}{|D|}\int_{D}\varphi dx$, then we can get
\begin{align*}\int_{D}p(t)\varphi
dx\!=\!\!\int_{D}(v(t)-v_0)\cdot\mathcal{B}(\varphi)dx\!+\!\!\int_0^t\!\!\!\int_{D}(H_1\!+\!H_2):\nabla\mathcal{B}(\varphi)dxdr
\!-\!\!\int_0^t\!\!\!\int_{D}\!\Phi\cdot\mathcal{B}(\varphi)dx
dW(r).
\end{align*}
Hence, we have
$$p(t)=\mathcal{B}^*(v(t)-v_0)+\int_0^t(\nabla\mathcal{B})^*(H_1+H_2)dr-\int_0^t\mathcal{B}^*\Phi dW(r),$$
where $\mathcal{B}^*$ denotes the adjoint of $\mathcal{B}$ with
respect to the $L^2(D)$ inner product.

Since $\theta:=\min\{2,\alpha\}$, using the continuity of
$\mathcal{B}^*$
on $L^2(D)$, $(\nabla B)^*$ on $L^\alpha(D)$ and the 
Burkholder-Davis-Gundy inequality, one has

\begin{equation}\label{3.2}
\begin{split}
E\!\left(\sup_{(0,T)}\!\int_{D}|p|^\theta \right)\!&\lesssim\!
E\!\left(\sup_{(0,T)}\!\int_{D}\!|\mathcal{B}^*(v(t)-v_0)|^2
\!+\!\int_Q\!|(\nabla\mathcal{B})^*(H_1+H_2)|^\alpha\!+\!\int_0^t\!(\mathcal{B}^*\Phi)^2dW(r)\right)\\
&\lesssim\!
E\!\left(1\!+\!\sup_{(0,T)}\int_{D}|v|^2\!+\!|v_0|^2+\int_Q|H_1+H_2|^\alpha
\!+\!\int_0^T\|\Phi\|^2_{L_2({U},L^2(D))}dt\right).
\end{split}
\end{equation}
Then $p\in L^\theta(\Omega; L^\infty(0,T;L^\theta(D)))$.

Let $\Delta_D^{-2}$ be the solution operator to the bi-Laplace
equation with respect to zero boundary values for function and
gradient. Let $p_0=\Delta\Delta_D^{-2}\Delta p$ and $p_h=p-p_0$.
Using the continuity of the operator $\Delta\Delta_D^{-2}\Delta$
from $L^\theta(D)$ to $L^\theta(D)$ (see \cite{RM}), we have
\begin{equation}\label{3.3}
\begin{split}
&E\left(\sup_{t\in(0,T)}\int_{D}|p_0|^\theta dx\right)\\
&\lesssim
E\left(1+\sup_{t\in(0,T)}\int_{D}|v|^2dx+\int_{D}|v_0|^2dx+\int_Q|H_1+H_2|^\alpha
dxdt+\int_0^T\|\Phi\|^2_{L_2({U},L^2(D))}dt\right).
\end{split}
\end{equation}
\begin{equation}\label{3.4}
\begin{split}
&E\left(\sup_{t\in(0,T)}\int_{D}|p_h|^\theta dx\right)\\
&\lesssim
E\left(1+\sup_{t\in(0,T)}\int_{D}|v|^2dx+\int_{D}|v_0|^2dx+\int_Q|H_1+H_2|^\alpha
dxdt+\int_0^T\|\Phi\|^2_{L_2({U},L^2(D))}dt\right).
\end{split}
\end{equation}

Note that $p_0(t)\in \Delta W_0^{2,\theta}(D)$ is uniquely
determined as the solution to the following equation:
\begin{equation}\label{3.5}
\begin{split}
\int_{D}p_0(t)\Delta\varphi
dx=\int_0^t\!\!\!\int_{D}(H_1+H_2):\nabla^2\varphi
dxdr-\int_0^t\!\!\!\int_{D}\Phi\cdot\nabla\varphi dx dW(r),
\end{split}
\end{equation}
for all $\varphi\in C_0^\infty(D)$.

From \cite{RM}, we know that $p_1\in\Delta W_0^{2,\alpha_1}(D)$ and
$p_2\in \Delta W_0^{2,\alpha_2}(D)$ are the unique solutions
(defined $\mathbb{P}\otimes \mathcal{L}^1$-a.e.) such that
\begin{equation}\label{3.6}
\begin{split}
\int_{D}p_1(t)\Delta\varphi dx=\int_DH_1:\nabla^2\varphi dx,
\end{split}
\end{equation}
\begin{equation}\label{3.7}
\begin{split}
\int_{D}p_2(t)\Delta\varphi dx=\int_DH_2:\nabla^2\varphi dx,
\end{split}
\end{equation}
for all $\varphi\in C_0^\infty(D)$. Then we have
$$\int_{D}(p_1(t)+p_2(t))\Delta\varphi dx=\int_D(H_1+H_2):\nabla^2\varphi dx,$$
for all $\varphi\in C_0^\infty(D)$ and $p_1+p_2\in\Delta
W_0^{2,\tilde{q}}(D)$. From Lemma 2.3 in \cite{JW}, it follows that
$$\int_D|p_1|^{\alpha_1}dx\leq c \int_D|H_1|^{\alpha_1}dx,$$
$$ \int_D|p_2|^{\alpha_2}dx\leq c \int_D|H_2|^{\alpha_2}dx,$$
$$ \int_D|p_1+p_2|^{\alpha}dx\leq c \int_D|H_1+H_2|^{\alpha}dx\quad \mathbb{P}\otimes\mathcal{L}^1-a.e..$$
These imply
$$E\left(\int_Q|p_1|^{\alpha_1}dxdt\right)^\beta\leq cE\left(\int_Q|H_1|^{\alpha_1}dxdt\right)^\beta,$$
$$E\left(\int_Q|p_2|^{\alpha_2}dxdt\right)^\beta\leq cE\left(\int_Q|H_2|^{\alpha_2}dxdt\right)^\beta,$$
$$E\left(\int_0^T\int_{D^\prime}|\nabla p_2|^{\alpha_2}dxdt\right)^\beta\leq cE\left(\int_Q|H_2|^{\alpha_2}+|\nabla H_2|^{\alpha_2}dxdt\right)^\beta,$$
$$E\left(\int_Q|p_1+p_2|^{\alpha}dxdt\right)\leq cE\left(\int_Q|H_1+H_2|^{\alpha}dxdt\right).$$

Let $p_\Phi:=p_0(t)-\int_0^t(p_1+p_2)dr \in \Delta
W_0^{2,\theta}(D)$. From \eqref{3.5}, \eqref{3.6} and \eqref{3.7},
it follows that $p_\Phi$ is the unique solution to
$$\int_{D}p_\Phi(t)\Delta\varphi dx=\int_0^t\!\!\!\int_{D}\Phi\cdot\nabla\varphi dx dW(r),$$
for all $\varphi\in C_0^\infty(D)$. Since $p_{\Phi}(t)\in \Delta
W_0^{2,\theta}(D)$, by Weyl's Lemma, for all $\varphi\in
C_0^\infty(D)$, we have
$$\int_{D}p_\Phi(t)\varphi dx=\int_0^t\!\!\!\int_{D}\Phi\cdot\nabla(\Delta^{-2}\Delta\varphi) dx dW(r).$$
Then $p_\Phi=\int_0^t\mathcal{D}^*\Phi dW(r)$,
$\mathbb{P}\otimes\mathcal{L}^{n+1}$-a.e., where
$\mathcal{D}=\nabla\Delta_D^{-2}\Delta:L^2(D)\rightarrow
W_0^{1,2}(D)$, $\mathcal{D}^*:L^2(D)\rightarrow L^2(D)$. Using the
 Burkholder-Davis-Gundy inequality, we obtain
\begin{equation}\label{3.8}
\begin{split}
E\left(\sup_{t\in(0,T)}\int_{D}|p_\Phi|^2 dx\right)
&\leq cE\left(\sup_{t\in(0,T)}\|\mathcal{D}^*\Phi\|^2_{L_2(U,L^2(D))}dt\right)\\
&\leq cE\left(\sup_{t\in(0,T)}\|\Phi\|^2_{L_2(U,L^2(D))}dt\right).
\end{split}
\end{equation}

Finally, we can infer that
$\tilde{p}_0(t):=p_\Phi(t)+\int_0^t(p_1+p_2)dr$ solves \eqref{3.5}
and there holds $\tilde{p}_{\Phi}(t)\in \Delta W_0^{2,\theta}(D)$
which implies $p_0(t):=p_\Phi(t)+\int_0^t(p_1+p_2)dr$. Then, we get
the equation claimed in (2) of Theorem \ref{theorem3.1}.

\end{proof}

\begin{Corollary}
Let the assumptions of Theorem 3.1 be satisfied. There exists
$\Phi_p\in L^2(\Omega; L^\infty(0,T;L_2({U},L^2_{loc}(D))))$
progressively measurable such that
$$\int_{D}p_\Phi(t){\rm{div}}\varphi dx=\int_0^t\!\!\!\int_{D}\Phi_p\cdot\varphi dx dW(r),\; \forall \varphi\in C_0^\infty(D).$$
Let $D'\subset\subset D$, then $\Phi_p$ satisfies $\|\Phi_p
e_k\|_{L^2(D')}\leq c(D')\|\Phi e_k\|_{L^2(D)}$, $\forall k$, that
is, it holds that $\mathbb{P}\otimes \mathcal{L}^1$-a.e.
$$\|\Phi_p \|_{L_2({U},L^2(D'))}\leq c(D')\|\Phi\|_{L_2({U},L^2(D))}.$$
If we assume that $\Phi$ satisfies \eqref{2.1} and \eqref{2.2}, then there holds
$$\|\Phi_p(v_1)- \Phi_p(v_2)\|_{L_2({U},L^2(D'))}\leq c(D')\|v_1-v_2\|_{L^2(D)},\,\forall v_1, v_2\in
L^2(D).$$
\end{Corollary}

\begin{proof}
From the proof of Theorem 3.1, it follows that
\begin{equation}\label{3.9}
\begin{split}
\int_{D}p_\Phi(t){\rm{div}}\varphi dx&=\int_0^t\!\!\!\int_{D}\Phi\cdot\nabla(\Delta^{-2}\Delta{\rm{div}}\varphi) dx dW(r)\\
&=\sum_k\int_0^t\!\!\!\int_{D}\Phi e_k\cdot\nabla(\Delta^{-2}\Delta{\rm{div}}\varphi) dx d\beta_k\\
&=\sum_k\int_0^t\!\!\!\int_{D}\nabla\Delta\Delta^{-2}{\rm{div}}\Phi e_k \cdot\varphi dxd\beta_k\\
&=\int_0^t\!\!\!\int_{D}\nabla\Delta\Delta^{-2}{\rm{div}}\Phi
\cdot\varphi dxdW(r),
\end{split}
\end{equation}
for all $\varphi\in C_0^\infty(D)$. Let
$\Phi_p=\nabla\Delta\Delta^{-2}\rm{div}\Phi$. Then we can get the
first claim. By using the local regularity theory for the bi-Laplace
equation in \cite{RM}, we can prove the rest results.
\end{proof}

\section{The approximated System}
Let us consider the following approximate system:
\begin{equation}\label{4.1}
\begin{cases}
du+\nabla\cdot(u\otimes u-S+p\text{I})dt+\varepsilon|u|^{\tilde{q}-2}udt=fdt+\Phi(u)dW,\\
u|_{t=0}=u_0,
\end{cases}
\end{equation}
for $\varepsilon>0$, depending on the law $\mu_0$ on
$L^2_{\text{div}}(D)$ and $\mu_f$ on $L^2(Q)$.

Assume that $f$ is adapted to 
$\mathscr{F}_t$ (otherwise enlarge it) and $f\in
L^2(\Omega;L^2_{\text{div}}(Q))$ with $\mu_f=\mathbb{P}\circ f^{-1}$
and $u_0\in L^2(\Omega;L^2_{\text{div}}(D))$ with
$\mu_0=\mathbb{P}\circ u_0^{-1}$. For the purpose of control the
nonlinear term $u\otimes u:\nabla u$, we add the term
$\varepsilon|u|^{\tilde{q}-2}u$  and choose $\tilde{q}\geq\max
\{2q',3\}$ such that the solution is an admissible test function.
Notice that $\frac{2}{\tilde{q}}+\frac{1}{p}\leq1$ and
$\frac{1}{\tilde{q}-1}+\frac{1}{2}\leq1$.  Let
$$V_{q,\tilde{q}}=L^2(\Omega;L^\infty(0,T;L^2(D)))\cap L^{\tilde{q}}(\Omega\times Q; \mathbb{P}\otimes \mathcal{L}^{n+1})\cap L^q (\Omega; L^q(0,T;W_{0,\text{div}}^{1,q}(D))).$$

From the appendix of \cite{JMJNMRMR}, we know that there exist a
sequence $\{\lambda_k\}\subset \mathbb{R}$ and a sequence of
functions $\{w_k\}\subset W_{0,\text{div}}^{\ell,2}(D)$, $\ell\in
\mathbb{N}$ such that

 (a) $w_k$ is an eigenvector
to the eigenvalue $\lambda_k$ of the Stokes-operator in the sense
that
$$\langle w_k, \varphi\rangle_{W_0^{\ell,2}}=\lambda_k\int_Dw_k\cdot\varphi dx,\, \forall \varphi\in W_{0,\text{div}}^{\ell,2}(D),$$

(b) $\int_D w_kw_m dx=\delta_{km},\, \forall k,m\in \mathbb{N},$

(c) $1\leq\lambda_1\leq\lambda_2\leq\cdots$ and
$\lambda_k\rightarrow\infty$,

(d)
$\langle\frac{w_k}{\sqrt{\lambda_k}},\frac{w_m}{\sqrt{\lambda_m}}\rangle_{W_0^{\ell,2}}=\delta_{km}$,\,
$\forall k,m\in \mathbb{N},$

 (e) $\{w_k\}$ is a basis of $W_{0,\text{div}}^{\ell,2}(D).$

Now, we use Galerkin approximation to separate space and time. Then
approximate equations \eqref{4.1} becomes an ordinary stochastic
differential equation. By using the classical existence theorems for
SDEs from \cite{LA}, \cite{AF1} and \cite{AF2}, we can prove the
existence of approximated solution. To this end, choosing
$\ell>1+\frac{n}{2}$, such that $W_0^{\ell,2}(D)\hookrightarrow
W^{1,\infty}(D)$. We are finding an approximated solution:
$$u^N=\sum_{k=1}^N c_k^Nw_k=C^N\cdot w^N,$$ where
$C^N=(c_i^N):\Omega\times (0,T)\rightarrow \mathbb{R}^N$ and
$w^N=(w_1,w_2,\cdots,w_N)$.

Let $\mathcal{P}^N: L^2_{\text{div}}(D)\rightarrow
\mathcal{X}:=\text{span}\{w_1,w_2,\cdots,w_N\}$ be the orthogonal
projection, i.e.,
$$\mathcal{P}^Nv=\sum_{k=1}^N\langle v,w_k\rangle_{L^2}\cdot w_k.$$
Therefore,  we would like to solve the system
\begin{equation}\label{4.2}
\begin{split}
&\int_{D}du^N\cdot
w_kdx+\int_{D}S(x,t,\mathbb{D}(u^N)):\mathbb{D}(w_k)dxdt+\varepsilon\int_{D}|u^N|^{\tilde{q}-2}u^N\cdot
w_k dxdt\\&=\int_{D}u^N\otimes u^N :\nabla w_kdxdt
+\int_{D}f\cdot w_k  dxdt+\int_{D}\Phi(u^N)\cdot w_k dxdW(t),\\
&u^N(0)=\mathcal{P}u_0,
\end{split}
\end{equation}
$\mathbb{P}$-a.s. for $k=1,2,\cdots, N$ and for a.e. $t$.

Assume that $W^N(s)=\sum_{k=1}^N \beta_k e_k(s)=\beta^N(s)\cdot
e^N$. Then it turned out to solving the following ordinary
stochastic differential equation:
\begin{equation}\label{4.3}
\begin{cases}
dC^N=A(t,C^N)dt+B(C^N)d\beta_t^N,\\
C^N(0)=C_0,
\end{cases}
\end{equation}
where
\begin{align*}
&A(t,C^N)\!=\!\left(-\int_{D}S(x,t,C^N\cdot \mathbb{D}(w^N)):\mathbb{D}(w_k)dx+\int_{D}(C^N\cdot w^N)\otimes(C^N\cdot w^N):\nabla w_kdx\right)_{k=1}^N\\
&\qquad\qquad\quad-\left(\varepsilon\int_{D}|C^N\cdot
w^N|^{\tilde{q}-2}(c^N\cdot w^N)\cdot w_k
dxdt\right)_{k=1}^N+\left(\int_{D}f\cdot w_k dx\right)_{k=1}^N,\\
&B(C^N)\!=\!\left(\int_{D}\Phi(C^N\cdot w^N)e_l\cdot
w_kdx\right)_{k,l=1}^N,\\
&C_0=\left(\langle v_0,w_k\rangle_{L^2(D)}\right)_{k=1}^N.
\end{align*}

In order to make use of the classical existence theorems for SDEs,
we need to prove that $A$ and $B$ satisfy globally Lipschitz
continuous condition and growth condition in the following. Note
that
\begin{equation*}
\begin{split}
(A&(t,C^N)-A(t,\hat{C}^N))\cdot(C^N-\hat{C}^N)\\
&=-\int_{D}(S(x,t,\mathbb{D}(u^N))-S(x,t,\mathbb{D}(\hat{u}^N))):(\mathbb{D}(u^N)-\mathbb{D}(\hat{u}^N))dx\\
&\quad+\int_{D}(u^N\otimes u^N-\hat{u}^N\otimes \hat{u}^N):(\mathbb{D}(u^N)-\mathbb{D}(\hat{u}^N))dx\\
&\quad-\varepsilon\int_{D}(|u^N|^{\tilde{q}-2}u^N-|\hat{u}^N|^{\tilde{q}-2}\hat{u}^N)(u^N-\hat{u}^N)dx\\
&\leq\int_{D}(u^N\otimes u^N-\hat{u}^N\otimes
\hat{u}^N):(\mathbb{D}(u^N)-\mathbb{D}(\hat{u}^N))dx.
\end{split}
\end{equation*}
Here we have used the monotonicity assumption \eqref{1.4}. If
$|C^N|\leq R$ and $|\hat{C}^N|\leq R$, then
$$(A(t,C^N)-A(t,\hat{C}^N))\cdot(C^N-\hat{C}^N)\leq c(R,N)|C^N-\hat{C}^N|^2.$$
This implies weak monotonicity in the sense of (3.1.3) in
\cite{CPMR} by using Lipschitz continuity $B$ for $C^N$, cf
\eqref{2.1} and \eqref{2.2}. By virtue of $\int_{D}u^N\otimes u^N:
\mathbb{D}(u^N)dx=0$, \eqref{1.3} and H\"{o}lder's inequality, we
have
\begin{align*}
A(t,C^N)\cdot C^N&=-\int_{D}S(x,t,\mathbb{D}(u^N)):\mathbb{D}(u^N)dx
-\varepsilon\int_{D}|u^N|^{\tilde{q}}dx+\int_{D}f(t)\cdot u^N dx\\
&\leq\int_{D}\eta_2 dx+\int_{D}f(t)\cdot u^N dx\\
&\leq c(1+\|f(t)\|_2\|v^N\|_2)\\
&\leq c(1+\|f(t)\|_2)(1+\|C^N\|_2).
\end{align*}
By using H\"{o}lder's inequality and \eqref{2.1}, one has
$$A(t,C^N)\cdot C^N+|B(C^N)|^2\leq c(1+\|f(t)\|_2)(1+|C^N|^2).$$
Since $\int_0^t(1+\|f(t)\|_2)dt<\infty$ $\mathbb{P}$-a.s., this
yields weak growth condition in the sense of (3.1.4) in \cite{CPMR}.
Then we obtain a unique strong solution $C^N\in
L^2(\Omega;C^0([0,T]))$ to the SDE \eqref{4.3}.

Next, we will get a priori estimate.
\begin{Lemma}\label{lemma4.1}
Under the assumption of \eqref{1.2}, \eqref{1.3} and \eqref{1.4} with $q\in (1,\infty)$, \eqref{2.1}, \eqref{2.2}, $\tilde{q}\geq\{2q',3\}$ and
\begin{align}\label{4.4}
\int_{L^2_{\rm{div}}(D)}\|v\|^2_{L^2(D)}d\mu_0(v)<\infty,\quad
\int_{L^2(Q)}\|\mathbf{g}\|^2_{L^2(Q)}d\mu_f(\mathbf{g})<\infty,
\end{align}
then there holds uniformly in $N$:
\begin{align*}
&E\left(\sup_{t\in(0,T)}\int_{D}|u^N(t)|^2dx+\int_Q|\nabla u^N|^qdxdt+\varepsilon\int_Q|u^N|^{\tilde{q}}dxdt\right)\\
&\quad\leq
c\left(1+\int_{L^2_{\rm{div}}(D)}\|v\|^2_{L^2(D)}d\mu_0(v)+\int_{L^2(Q)}\|\mathbf{g}\|^2_{L^2(Q)}d\mu_f(\mathbf{g})\right),
\end{align*}
where $c$ is independent of $\varepsilon$.
\end{Lemma}

\begin{proof}
Since $du^N=\sum_{k=1}^N dc_k^N\cdot w_k$, $\int_{D}u^N\otimes u^N:
\mathbb{D}u^Ndx=0$,  $\int_Dw_kw_mdx=\delta_{km}, \forall k,m\in
\mathbb{N}$, and \begin{align*}
dc_k^N&=-\int_{D}S(x,t,\mathbb{D}(u^N)):\mathbb{D}(w_k)dxdt-\varepsilon\int_{D}|u^N|^{\tilde{q}-2}u^N\cdot
w_k dxdt\\&\quad+\int_{D}u^N\otimes u^N :\nabla w_kdxdt
+\int_{D}f\cdot w_k dxdt +\int_{D}\Phi(u^N)\cdot w_k dxdW^N(t),
\end{align*}
It\^{o}'s formula $f(X)=\frac{1}{2}|X|^2$ yields
\begin{align}\label{4.5}
\notag
\frac{1}{2}\|u^N(t)\|^2_{L^2(D)}&=\frac{1}{2}\|C^N(0)\|^2_{L^2(D)}+\sum_{k=1}^N\int_0^t\!\!\!\int_D
C_k^N d(C_k^N)(r)dxdr
+\frac{1}{2}\sum_{k=1}^N\int_0^t\!\!\!\int_Dd\langle C_k^N\rangle(r)dxdr\\
\notag&=\frac{1}{2}\|\mathcal{P}^Nu_0\|^2_{L^2(D)}-\int_0^t\!\!\!\int_DS(x,r,\mathbb{D}(u^N)):\mathbb{D}(u^N)dxdr-\varepsilon\int_0^t\!\!\!\int_D
|u^N|^{\tilde{q}}dxdr\\
&\quad+\int_0^t\!\!\!\int_Df\cdot u^Ndxdr+\int_0^t\!\!\!\int_D
u^N\cdot
\Phi(u^N) dxdW^N(r)\\
\notag&\quad+\frac{1}{2}\sum_{i=1}^N\int_0^t\!\!\!\int_D|\Phi(u^N)e_i|^2dxdr.
\end{align}
From \eqref{1.3}, \eqref{4.5} and Korn's inequality, it follows that
\begin{align*}
&E\left(\int_{D}|u^N(t)|^2dx+\int_0^t\!\!\!\int_D|\nabla u^N|^qdxdr+\varepsilon\int_0^t\!\!\!\int_D|u^N|^{\tilde{q}}dxdr\right)\\
&\leq c\left[1+I_1+I_2+I_3+E\left(\|v_0\|^2_{L^2(D)}\right)\right],
\end{align*}
where
\begin{align*}
I_1&=E\left(\int_0^t\!\!\!\int_Df\cdot u^Ndxdr\right),\\
I_2&=E\left(\int_0^t\!\!\!\int_Du^N\cdot \Phi(u^N)dxdW^N(r)\right),\\
I_3&=E\left(\sum_{i=1}^N\int_0^t\!\!\!\int_D|\Phi(u^N)e_i|^2dxdr
\right).
\end{align*}
By using Young's inequality, we have
$$I_1\leq\delta
E\left(\int_0^t\!\!\!\int_{D}|u^N(t)|^2dxdr\right)+c(\delta)E\left(\int_0^t\!\!\!\int_D|f|^2dxdr\right)~
\mbox{ for }~\forall\delta>0.$$ It is clear that $I_2=0$. Thanks to
\eqref{2.1}, we deduce that
\begin{equation*}
\begin{split}
I_3&\leq E\left(\sum_{i=1}^N\int_0^t\!\!\!\int_D|g_i(u^N)|^2dxdr\right)\\
&\leq E\left(1+\int_0^t\!\!\!\int_D|u^N|^2dxdr\right).
\end{split}
\end{equation*}
Then, by interchanging the time-integral and the expectation value
and using Gronwall's inequality, we obtain
\begin{equation}\label{4.6}
\begin{split}
&E\left(\sup_{t\in(0,T)}\int_{D}|u^N(t)|^2dx\right)+E\left(\int_{Q}|\nabla u^N|^qdxdt\right)\\
&\quad\leq cE\left(1+\int_{D}|u_0|^2dx+\int_Q|f|^2dxdt\right).
\end{split}
\end{equation}

Similarly, we have
\begin{equation}\label{4.7}
\begin{split}
E\left(\sup_{t\in(0,T)}\int_{D}|u^N(t)|^2dx\right)&\leq
cE\left(1+\int_{D}|u_0|^2dx+\int_Q|f|^2dxdt
+\int_0^T\!\!\!\int_{D}|u^N|^2dxdt\right)\\
&\quad+E\left(\sup_{t\in(0,T)}\left|\int_0^T\!\!\!\int_Du^N\cdot
\Phi(u^N)dxdW^N(t)\right|\right).
\end{split}
\end{equation}
Using Burkholder-Davis-Gundy inequality, H\"{o}lder's inequality,
Young's inequality and \eqref{2.1}, one has
\begin{align*}
&E\left(\sup_{t\in(0,T)}\left|\int_0^t\!\!\!\int_Du^N\cdot \Phi(u^N)dxdW^N(r)\right|\right)\\
&=E\left(\sup_{t\in(0,T)}\left|\sum_{i}\int_0^t\!\!\!\int_Du^N\cdot \Phi(u^N)e_idxd\beta_i(r)\right|\right)\\
&=E\left(\sup_{t\in(0,T)}\left|\sum_{i}\int_0^t\!\!\!\int_Du^N\cdot g_i(u^N)dxd\beta_i(r)\right|\right)\\
&\leq cE\left[\int_0^T\sum_{i}\left(\int_Du^N\cdot g_i(u^N)dx\right)^2dt\right]^{\frac{1}{2}}\\
&\leq cE\left[\int_0^T\left(\sum_{i}\int_D|u^N|^2dx \cdot\int_D|g_i(u^N)|^2dx\right)dt\right]^{\frac{1}{2}}\\
&\leq \delta E\left(\sup_{t\in(0,T)}\int_D|u^N|^2dx\right)+c(\delta)E\left(1+\int_0^T\int_D|u^N|^2dxdt\right).\\
\end{align*}
For $\delta$ sufficiently small, this together with \eqref{4.6}
yield Lemma \ref{lemma4.1}.
\end{proof}

\begin{Lemma}
Assume that \eqref{1.2}-\eqref{1.4} with $q\in (1,\infty)$,
\eqref{2.1}, \eqref{2.2}, $\tilde{q}\geq\{2q',3\}$ and \eqref{4.4}
hold. Then

 $(1)$ There exists a martingale weak solution
$((\overline{\Omega},\overline{\mathscr{F}},\overline{\mathscr{F}}_t,\overline{\mathbb{P}}),\overline{u},\overline{u}_0,\overline{f},\overline{W})$
to \eqref{4.1} in the sense that:

 $(\rm a)$
$(\overline{\Omega},\overline{\mathscr{F}},\overline{\mathscr{F}}_t,\overline{\mathbb{P}})$
is a stochastic basis with a complete right-continuous filtration;

$(\rm b)$ $\overline{W}$ is an
$\overline{\mathscr{F}}_t$-cylindrical Wiener process;

$(\rm c)$ $\overline{u}\in \overline{V}_{q,\tilde{q}}$ is
progressively measurable, where
$$\overline{V}_{q,\tilde{q}}=L^2(\overline{\Omega};L^\infty(0,T;L^2(D)))\cap L^{\tilde{q}}
(\overline{\Omega}\times Q; \overline{\mathbb{P}}\otimes
\mathcal{L}^{n+1})\cap L^q
(\overline{\Omega};L^q(0,T;W_{0,\rm{div}}^{1,q}(D)));$$

$(\rm d)$ $\overline{u}_0\in L^2(\overline{\Omega}; L^{2}(D))$ with
$\mu_0=\overline{\mathbb{P}}\circ \overline{u}_0^{-1}$;

$(\rm e)$ $\overline{f}\in L^2(\overline{\Omega}; L^{2}(Q))$ is
adapted to $\overline{\mathscr{F}}_t$ and
$\mu_f=\overline{\mathbb{P}}\circ \overline{f}^{-1}$;

$(\rm f )$ $\forall \varphi\in C_{0,\rm{div}}^{\infty}(D)$ and
$\forall t\in [0,T]$, it holds that $\overline{\mathbb{P}}$-a.s.
\begin{equation*}
\begin{split}
&\int_{D}\!(\overline{u}(t)\!-\!\overline{u}_0)\cdot\varphi
dx\!+\!\varepsilon\!\!\int_0^t\!\!\!\int_{D}|\overline{u}|^{\tilde{q}-2}\overline{u}\cdot\varphi
dxdr\!-\!\!\int_0^t\!\!\!\int_{D}\overline{u}\otimes \overline{u}
:\mathbb{D}(\varphi)\!+\!S(x,r,\mathbb{D}(\overline{u})):\mathbb{D}(\varphi)
dxdr\\
&=\int_0^t\!\!\!\int_{D}\overline{f}\cdot \varphi dx dr
+\int_0^t\!\!\!\int_{D}\Phi(\overline{u})\cdot \varphi
dxd\overline{W}(r),
\end{split}
\end{equation*}

$(2)$ There holds
\begin{equation*}
\begin{split}
\overline{E}&\left(\sup_{t\in(0,T)}\int_{D}|\overline{u}(t)|^2dx+\int_Q|\nabla \overline{u}|^qdxdt+\varepsilon\int_Q|\overline{u}|^{\tilde{q}}dxdt\right)\\
&\leq
c\left(1+\int_{L^2_{\rm{div}}(D)}\|v\|^2_{L^2(D)}d\mu_0(v)+\int_{L^2(Q)}\|\mathbf{g}\|^2_{L^2(Q)}d\mu_f(\mathbf{g})\right),
\end{split}
\end{equation*}
where $c$ is independent of $\varepsilon$.
\end{Lemma}

\begin{proof}
Let $\mathcal{S}(u)=\varepsilon|u|^{\tilde{q}-2}u$, from Lemma 4.1,
we know that there exist functions $u\in V_{q,\tilde{q}}$ and
functions $\tilde{\mathcal{S}}$ and $\tilde{S}$, such that
\begin{align}\label{4.8}
u^N\rightharpoonup u &~\mbox{ in }~ L^q
(\Omega;L^q(0,T;W_{0,\rm{div}}^{1,q}(D))),\\
\label{4.9} u^N\rightharpoonup u& ~\mbox{ in }~
L^{\tilde{q}}(\Omega;L^{\tilde{q}}(Q)) \\
\label{4.10} \mathcal{S}(u^N)\rightharpoonup \tilde{\mathcal{S}}&
~\mbox{ in }~
L^{\tilde{q}'}(\Omega;L^{\tilde{q}'}(Q))\\
\label{4.11} S(x,t,\mathbb{D}(u^N))\rightharpoonup \tilde{S}&
~\mbox{ in }~
L^{q'}(\Omega; L^{q'}(Q)) \\
\label{4.12} S(x,t,\mathbb{D}(u^N))\rightharpoonup \tilde{S}&
~\mbox{ in }~ L^{q'} (\Omega;L^{q'}(0,T;W_{0,\rm{div}}^{-1,q'}(D)))
\\
\label{4.13} u^N\otimes u^N\rightharpoonup \tilde{U}& ~\mbox{ in }~
{L^{\tilde{q}/2}(\Omega;L^{\tilde{q}/2}(Q))} \\
\label{4.14} \Phi(u^N)\rightharpoonup \tilde{\Phi}& ~\mbox{ in }~
L^2 (\Omega;L^2(0,T;L_2({U},L^2(D))))
\end{align}

In order to prove
\begin{equation}\label{4.15}
\begin{split}
\tilde{U}=u\otimes u,\quad \tilde{\Phi}=\Phi(u),
\end{split}
\end{equation}
we will use some compactness arguments similar to the ideas from
\cite[Sec.4]{MH}. Let $\mathcal{P}^N_{\ell}$
denotes the projection from $W_{0,\rm{div}}^{\ell,2}(D)$
into $\mathcal{X}_N$. By using \eqref{4.2}, we have
\begin{align*}
&\int_{D}u^N\cdot\varphi dx=\int_{D}u^N(t)\cdot\mathcal{P}^N_{\ell}\varphi dx\\
&=\int_{D}u_0\cdot\mathcal{P}^N_{\ell}\varphi
dx-\int_0^t\!\!\!\int_{D}(H^N_1+H^N_2):\nabla\mathcal{P}^N_{\ell}\varphi
dxdr+\int_0^t\!\!\!\int_{D}\Phi(u^N)\cdot\mathcal{P}^N_{\ell}\varphi
dx dW^N(r),
\end{align*}
where
\begin{align*}
H_1^N&:=S(x,t,\mathbb{D}(u^N)),\\
H_2^N&:=\nabla\Delta^{-1}f-\nabla\Delta^{-1}\mathcal{S}(u^N)-u^N\otimes
u^N.
\end{align*}

From Lemma 4.1, \eqref{1.2}-\eqref{1.4} and the fact
$\mathcal{S}=\varepsilon|u|^{\tilde{q}-2}u$, it follows that
\begin{equation}\label{4.16}
\begin{split}
H_1^N+H_2^N\in L^{q_0}(\Omega\times Q;\mathbb{P}\otimes
\mathcal{L}^{n+1}), \;
q_0:=\min\left\{q',\frac{\tilde{q}}{2},\tilde{q}'\right\}>1,
\end{split}
\end{equation}
uniformly in $N$. Let
$$\mathcal{H}(t,\varphi)=\int_0^t\!\!\!\int_D(H_1^N+H_2^N):\nabla\mathcal{P}_{\ell}^N\varphi dxdr,\; \varphi\in C^\infty_{0,\text{div}}(D).$$
By the fact $W^{\tilde{\ell},q^\prime_0}(D)\hookrightarrow
W_0^{\ell,2}(D)$ for $\tilde{\ell}\geq
\ell+n(1+\frac{2}{q^\prime_0})$ and \eqref{4.16}, we have
$$E\left(\|\mathcal{H}\|_{W^{1,q_0}(0,T;W_{\rm{div}}^{-\tilde{\ell},q_0}(D))}\right)\leq c.$$
For the stochastic term, using \eqref{2.1}, \eqref{2.2} and Lemma
4.1, for any $\vartheta>2$, one has
\begin{equation*}
\begin{split}
E\left(\left\|\int_s^t\Phi(u^N)dW^N(r)\right\|_{L^2(D)}\right)^\vartheta
\leq c(t-s)^{\frac{\vartheta}{2}}.
\end{split}
\end{equation*}
Thanks to the Kolmogorov continuity criterion \cite{GDJZ}, we can
infer that for any $\Lambda\in[0, 1/2)$,
\begin{equation*}
\begin{split}
E\left(\left\|\int_0^t\Phi(u^N)dW^N(r)\right\|_{C^\Lambda([0,T];L^2(D))}\right)
\leq c,
\end{split}
\end{equation*}
 Then
\begin{equation}\label{4.17}
\begin{split}
E\left(\left\|u^N\right\|_{C^\Lambda([0,T];W_{\rm{div}}^{-\tilde{\ell},q_0}(D))}\right)\leq
c,
\end{split}
\end{equation}
and
\begin{equation}\label{4.18}
\begin{split}
E\left(\left\|u^N\right\|_{W^{\lambda,q_0}(0,T;W_{\rm{div}}^{-\tilde{\ell},q_0}(D))}\right)\leq
c,
\end{split}
\end{equation}
for some $\lambda>0$. Note that an interpolation with
$L^{q_0}(0,T;W_{0,\rm{div}}^{1,q_0}(D))$
yields for some $\kappa>0$
\begin{equation}\label{4.19}
\begin{split}
E\left(\left\|u^N\right\|_{W^{\kappa,q_0}(0,T;L_{\rm{div}}^{q_0}(D))}\right)\leq
c.
\end{split}
\end{equation}

Now, we prepare the setup for our compactness method. Define the
path space of $(u^N,W,u_0,f)$ by
$$\mathcal{V}:=L^\gamma(0,T;L^\gamma(D))\times C([0,T],{U}_0)\times L^2_{\rm{div}}(D)\times L^2(Q).$$
Let us denote by $\mu_{u^N}$ the law of $u^N$ on
$L^\gamma(0,T;L^\gamma(D))$. By $\mu_{W}$, we denote the law of $W$
on $C([0,T],{U}_0)$. 
The joint law of $u^N$, $W$, $u_0$ and $f$ on $\mathcal{V}$ is
denoted by $\mu^N$.
\begin{Proposition}
The set $\{\mu^N|N\in\mathbb{N}\}$ is tight on $\mathcal{V}$.
\end{Proposition}
\begin{proof}
In order to prove the tightness of $\mu^N$, we need  the following
three steps.

{Step 1: Tightness of $\mu_{u^N}$.} 
On account of $L^{\tilde{q}}\hookrightarrow\hookrightarrow L^{q_0}$,
if $-\frac{n}{\gamma}<-\frac{n}{\tilde{q}}$, we can use Theorem 5.2
\cite{HA} to obtain
$$W^{\kappa,q_0}(0,T;L_{\rm{div}}^{q_0}(D))\cap V_{q,\tilde{q}}\hookrightarrow\hookrightarrow L^\gamma(0,T;L^\gamma_{\rm{div}}(D))$$
compactly for all $q_0<\gamma<\tilde{q}$. We consider the ball $B_R$
in the space $W^{\kappa,q_0}(0,T;L^{q_0}_{\rm{div}}(D))\cap
V_{q,\tilde{q}}$ and let $B_R^c$ be the complement of the ball.
Using Lemma 4.1 and \eqref{4.19}, we have
\begin{align*}
\mu_{u^N}(B_R^c)&=\mathbb{P}(\|u^N\|_{W^{\kappa,q_0}(0,T;L_{\rm{div}}^{q_0}(D))}+\|u^N\|_{V_{q,\tilde{q}}}\geq R)\\
&\leq
\frac{1}{R}E\left(\|u^N\|_{W^{\kappa,q_0}(0,T;L_{\rm{div}}^{q_0}(D))}+\|u^N\|_{V_{q,\tilde{q}}}\right)\leq\frac{c}{R}.
\end{align*}
Then, there exists $R(\eta)$ such that
$$\mu_{u^N}(B_{R(\eta)})\geq 1-\frac{\eta}{4},$$
for a fixed $\eta>0$. These yield the tightness of $\mu_{u^N}$.

{Step 2: Tightness of $\mu_W$.} We consider the ball $B_R$ in the
space $C([0,T];U_0)$ and let $B_R^c$ be the complement of the ball.
Then
\begin{align*}
\mu_W(B_R^c)&=\mathbb{P}(\|W\|_{C([0,T];U_0)}\geq R)\leq
\frac{1}{R}E\left(\|W\|_{C([0,T];U_0)}\right)\leq\frac{1}{R}E\left(\sup_{[0,T]}\|W(r)\|_{U_0}\right)\leq\frac{c}{R} .
\end{align*}
Then, there exists $R(\eta)$ such that
$$\mu_{W}(B_{R(\eta)})\geq 1-\frac{\eta}{4},$$
for a fixed $\eta>0$. These imply the tightness of $\mu_{W}$.

{Step 3: Tightness of $\mu_0,\mu_f$.} We consider the ball $B_R$ in
the space $L^2_{\text{div}}(D)$ and let $B_R^c$ be the complement of
the ball. Therefore
\begin{align*}
\mu_0(B_R^c)&=\mathbb{P}(\|u_0\|_{L^2_{\text{div}}(D)}\geq R)
\leq
\frac{1}{R}E\left(\|u_0\|_{L^2_{\text{div}}(D)}\right)\leq\frac{c}{R}.
\end{align*}
Then, there exists $R(\eta)$ such that
$$\mu_{0}(B_{R(\eta)})\geq 1-\frac{\eta}{4},$$
for a fixed $\eta>0$. These yield the tightness of $\mu_{0}$.

We consider the ball $B_R$
in the space $L^2(Q)$ and let $B_R^c$ be the complement of the ball.
Then we have
\begin{align*}
\mu_f(B_R^c)&=\mathbb{P}(\|u_f\|_{L^2(Q)}\geq R)
\leq
\frac{1}{R}E\left(\|u_f\|_{L^2(Q)}\right)\leq\frac{c}{R}.
\end{align*}
Then, there exists $R(\eta)$ such that
$$\mu_{f}(B_{R(\eta)})\geq 1-\frac{\eta}{4},$$
for a fixed $\eta>0$. These imply the tightness of $\mu_{f}$.

So we can find a compact subset $\mathcal{V}_\eta\subset\mathcal{V}$
such that $\mu^N(\mathcal{V}_\eta)\geq 1-\eta$. Thus,
$\{\mu^N|N\in\mathbb{N}\}$ is tight in the same space.

\end{proof}

Thanks to Prokhorov's Theorem in \cite{NISW}, we can infer that
$\mu^N$ is also relatively weakly compact. Then 
$\mu_n\to\mu$ weakly. By the Skorohod representation theorem in
\cite{NISW}, we know that the following result.
\begin{Proposition}
There exists a probability space
$(\overline{\Omega},\overline{\mathscr{F}},\overline{\mathbb{P}})$
with $\mathcal{V}$-valued Borel measurable random variables
$(\overline{u}^N,\overline{u}^N_0,\overline{f}^N,\overline{W}^N)$
and $(\overline{u},\overline{u}_0,\overline{f},\overline{W})$ such
that the following hold:

$\clubsuit$ The laws of
$(\overline{u}^N,\overline{u}^N_0,\overline{f}^N,\overline{W}^N)$
and $(\overline{u},\overline{u}_0,\overline{f},\overline{W})$ under
$\overline{\mathbb{P}}$ coincide with $\mu^N$ and $\mu$.

$\clubsuit$
\begin{align*}
\overline{u}^N\rightharpoonup \overline{u}&  ~\mbox{ \rm  in }~
L^{\gamma}(0,T;L^\gamma(D))\ \ \overline{\mathbb{P}}-\text{a.s.},\\
\overline{W}^N\rightharpoonup  \overline{W}& ~
 \mbox{ \rm  in }~   C([0,T],{U}_0)\ \ \overline{\mathbb{P}}-\text{a.s.},\\
\overline{u}_0^N\rightharpoonup \overline{u}_0&  ~\mbox{  \rm in }~
L^{2}(D)\ \ \overline{\mathbb{P}}-\text{a.s.},\\
 \overline{f}^N\rightharpoonup \overline{f}&  ~\mbox{ \rm  in }~
L^{2}(0,T;L^2(D))\ \ \overline{\mathbb{P}}-\text{a.s.}.
\end{align*}

$\clubsuit$ The convergence in \eqref{4.8} and \eqref{4.13} still
hold for the corresponding functions defined on
$(\overline{\Omega},\overline{\mathscr{F}},\overline{\mathbb{P}})$.
Moreover, we have
$$\int_{\overline{\Omega}}\left(\sup_{[0,T]}\|\overline{W}^N(t)\|^{\alpha}_{{U}_0}\right)
d\overline{\mathbb{P}}=\int_{\Omega}\left(\sup_{[0,T]}\|W(t)\|^{\alpha}_{{U}_0}\right)d\mathbb{P}\quad
\text{for all}\ \ \alpha<\infty.$$
\end{Proposition}

 By Vitali's convergence Theorem,
for all $\gamma<\tilde{q}$, we have
\begin{align}\label{4.20}
\overline{W}^N\rightarrow \overline{W}& ~\mbox{ in }~
L^2(\overline{\Omega};C([0,T],{U}_0)),\\
\label{4.21} \overline{u}^N\rightarrow \overline{u}& ~\mbox{ in }~
L^\gamma(\overline{\Omega}\times
Q;\overline{\mathbb{P}}\times\mathcal{L}^{n+1}), \\
\label{4.22} \overline{u}_0^N\rightarrow \overline{u}_0& ~\mbox{ in
}~ L^2(\overline{\Omega}\times
D;\overline{\mathbb{P}}\times\mathcal{L}^{n+1}), \\
\label{4.23} \overline{f}^N\rightarrow \overline{f}& ~\mbox{ in }~
L^2(\overline{\Omega}\times
Q;\overline{\mathbb{P}}\times\mathcal{L}^{n+1}),
\end{align}
after choosing a subsequence.

Now, we are going to show that the approximated equations also hold
on the new probability space. To this end, we define
\begin{align*}
\xi^N(t)&\!=\!\int_{D}(u^N(t)-u_0)\cdot\varphi dx-
\int_0^t\!\!\!\int_{D}u^N\otimes u^N :\nabla\mathcal{P}_N\varphi dxdr+{\int_0^t\!\!\!
\int_{D}\mathcal{S}(u^N)\cdot\mathcal{P}_N\varphi dxdr}\\
&\quad+\!\int_0^t\!\!\!\int_{D}S(x,r,\mathbb{D}(u^N)):\mathbb{D}(\mathcal{P}_N\varphi)\!
-\!f\!\cdot\! \mathcal{P}_N\varphi dx
dr\!-\!\int_0^t\!\!\!\int_{D}\Phi(u^N)\cdot\mathcal{P}^N_{\ell}\varphi
dx dW^N(r) ,
\end{align*}
\begin{align*}
Z^N\!=\!\int_0^T
\|\xi^N(t)\|^2_{W_{\rm{div}}^{-\tilde{\ell},q_0}(D))}dt.
\end{align*}
Of course
\begin{align*}
Z^N=0,\; \mathbb{P}-\text{a.s.}.
\end{align*}
Let \begin{align*}
\overline{\xi}^N(t)&\!=\!\int_{D}(\overline{u}^N(t)-\overline{u}^N_0)\cdot\varphi
dx- \int_0^t\!\!\!\int_{D}\overline{u}^N\otimes \overline{u}^N
:\nabla\mathcal{P}_N\varphi
dxdr+{\int_0^t\!\!\!\int_{D}
\mathcal{S}(\overline{u}^N)\cdot\mathcal{P}_N\varphi dxdr}\\
&\quad+\!\int_0^t\!\!\!\int_{D}S(x,r,\mathbb{D}(\overline{u}^N)):\mathbb{D}(\mathcal{P}_N\varphi)
\!-\!\overline{f}^N\!\!\cdot\! \mathcal{P}_N\varphi dx
dr\!-\!\int_0^t\!\!\!\int_{D}\Phi(\overline{u}^N)\cdot\mathcal{P}^N_{\ell}\varphi
dx d\overline{W}^N(r) ,
\end{align*}
\begin{align*}
Y^N\!=\!\int_0^T
\|\overline{\xi}^N(t)\|^2_{W_{\rm{div}}^{-\tilde{\ell},q_0}(D))}dt.
\end{align*}
We want to verify that
\begin{align*}
\overline{E}Y^{N}=0.
\end{align*}
To this end, we have the following Proposition:
\begin{Proposition}\label{prop4.2}
$Y^{N}=0,\;\overline{\mathbb{P}}-\mbox{\rm a.s.}$, that is,
$(\overline{u}^N,\overline{u}^N_0,\overline{f}^N,\overline{W}^N)$
satisfies the equation \eqref{4.1}.
\end{Proposition}
\begin{proof}The difficulty comes from $Z_n$ is not expressed as a
deterministic function of  $(u^N\!,W^N)$ because of the presence of
the stochastic integral. By Theorem 2.4 and Corollary 2.5 in
\cite{BH}, we can infer that
\begin{equation}\label{distrbution1}
\mathscr{L}(\overline{u}^N,\overline{u}^N_0,\overline{f}^N,\overline{W}^N,\overline{\xi}^N
)=\mathscr{L}(u^N,u^N_0,f^N,W^N,\xi^N).
\end{equation}
Here $\mathscr{L}(f)$ is the probability distribution of $f$. Note
that $Y^N$ is continuous as a function of $\overline{\xi}^N$. In
view of \eqref{distrbution1} and the continuity of $Y^N$, one
deduces that the distribution of $Y^N$ is equal to the distribution
of $Z^N$ on $\mathbb{R}_+$, that is,
\begin{equation}\label{distrbution2}
\overline{E}\phi(Y^N)=E\phi(Z^N),
\end{equation}
for any $\phi\in C_b(\mathbb{R}_+)$, where $C_b(X)$ is the space of
continuous bounded functions defined on $X$. Now, let
$\varepsilon>0$ be an arbitrary number and $\phi_\varepsilon\in
C_b(\mathbb{R}_+)$ defined by
\begin{equation*}
\phi_\varepsilon=
\begin{cases}
{y\over\varepsilon}, &0\le y<\varepsilon;\\
1, &y\ge \varepsilon.
\end{cases}
\end{equation*}
One can check that
\begin{equation*}
\begin{split}
\overline{\mathbb{P}}(Y^N\ge \varepsilon)&=\int_{\tilde{\Omega}}
1_{[\varepsilon,\infty]}Y^Nd\overline{\mathbb{P}}\le
\int_{\tilde{\Omega}}
1_{[0,\varepsilon]}{Y^N\over\varepsilon}d\overline{\mathbb{P}}+\int_{\tilde{\Omega}}
1_{[\varepsilon,\infty]}Y^Nd\overline{\mathbb{P}},
\end{split}
\end{equation*}
Hence by the definition of $\overline{E}\phi_\varepsilon(Y^N)$, we
can infer that
\begin{equation*}
\overline{\mathbb{P}}(Y^N\ge \varepsilon)\le
\overline{E}\phi_\varepsilon(Y^N),
\end{equation*}
which together with \eqref{distrbution2} imply that
\begin{equation*}
\overline{\mathbb{P}}(Y^N\ge \varepsilon)\le E\phi_\varepsilon(Z^N),
\end{equation*}
By the fact that $(u^N,u^N_0,f^N,W^N)$ satisfies the Galerkin
equation, from the above inequality, it holds that
\begin{equation}\label{distrbution3}
\overline{\mathbb{P}}(Y^N\ge \varepsilon)\le
E\phi_\varepsilon(Z^N)=0,
\end{equation}
for any $\varepsilon>0$. Since $\varepsilon>0$ is arbitrary, from
\eqref{distrbution3}, we can infer that
\begin{equation}\label{distrbution4}
Y^N=0,\; \overline{\mathbb{P}}-\mbox{\rm a.s.}.
\end{equation}
It follows from \eqref{distrbution4} that
$(\overline{u}^N,\overline{u}^N_0,\overline{f}^N,\overline{W}^N)$
satisfies the equation \eqref{4.1}.
\end{proof}


Since $\overline{W}^N$ has the same law as $W$, there exists a
collection of mutually independent real-valued
$\overline{\mathscr{F}}_t$-Wiener process
$\{\overline{\beta}_k^N\}_{k}$ such that $\overline{W}^N=\sum_k
\overline{\beta}_k^Ne_k$, i.e., there exists a collection of
mutually independent real-valued $\overline{\mathscr{F}}_t$-Wiener
process $\{\overline{\beta}_k\}_{k\geq1}$ such that
$\overline{W}=\sum_k \overline{\beta}_ke_k$. We denote
$\overline{W}^{N,N}\!\!:=\sum_{k=1}^N e_k\overline{\beta}_k^N.$
Proposition \ref{prop4.2} means the equations
\begin{equation}\label{4.28}
\begin{split}
&\int_{D}d\overline{u}^N\cdot
w_kdx+\int_{D}S(x,t,\mathbb{D}(\overline{u}^N)):\mathbb{D}(w_k)dxdt+\varepsilon\int_{D}|\overline{u}^N|^{\tilde{q}-2}\overline{u}^N\cdot
w_k dxdt\\&\quad=\int_{D}\overline{u}^N\otimes \overline{u}^N
:\nabla(w_k)dxdt
+\int_{D}\overline{f}\cdot w_k  dxdt+\int_{D}\Phi(\overline{u}^N)\cdot w_k dxd\overline{W}^{N,N}(t),\\
&\overline{u}^N(0)=\mathcal{P}^N\overline{u}_0,
\end{split}
\end{equation}
$(k=1,2,\cdots N)$ holds on the new probability space
$(\overline{\Omega},\overline{\mathscr{F}},\overline{\mathbb{P}})$.
At the same time, we have
\begin{align}\label{4.29}
\overline{u}^N\rightharpoonup \overline{u}& ~\mbox{ in }~   L^q
(\overline{\Omega};L^q(0,T;W_{0,\rm{div}}^{1,q}(D))),\\
\label{4.30}  \overline{u}^N\rightharpoonup \overline{u}&  ~\mbox{
in }~   L^{\tilde{q}}(\overline{\Omega};L^{\tilde{q}}(Q)),
\\
\label{4.31} \mathcal{S}(\overline{u}^N)\rightharpoonup
\mathcal{S}(\overline{u})& ~\mbox{ in }~
L^{\tilde{q}'}(\overline{\Omega};L^{\tilde{q}'}(Q)),
\\
\label{4.32} S(x,t,\mathbb{D}(\overline{u}^N)) \rightharpoonup
\tilde{\overline{S}}&  ~\mbox{ in }~
L^{q'}(\overline{\Omega};L^{q'}(Q)),
\\
\label{4.33} S(x,t,\mathbb{D}(\overline{u}^N)) \rightharpoonup
\tilde{\overline{S}}&  ~\mbox{ in }~
L^{q'}(\overline{\Omega};L^{q'}(0,T;W_{0,\rm{div}}^{-1,q'}(D))),\\
\label{4.34} \overline{u}^N\otimes \overline{u}^N\rightharpoonup
\overline{u}\otimes \overline{u}&  ~\mbox{ in }~
L^{\tilde{q}/2}(\overline{\Omega};L^{\tilde{q}/2}(Q)),\\
\label{4.35} \Phi(\overline{u}^N)\rightharpoonup \Phi(\overline{u})&
~\mbox{ in }~   L^2 (\overline{\Omega};L^2(0,T;L_2({U},L^2(D)))).
\end{align} By using \eqref{4.20}-\eqref{4.29}, one has
\begin{equation}\label{4.36}
\begin{split}
\int_{D}&(\overline{u}(t)-\overline{u}_0)\cdot \varphi
dx+\int_0^t\!\!\!\int_{D}\tilde{\overline{S}}:\nabla\varphi
dxdr+\int_0^t\!\!\!\int_{D}\mathcal{S}(\overline{u})\cdot \varphi
dxdr=\int_0^t\!\!\!\int_{D}\overline{u}\otimes \overline{u}
:\nabla\varphi dxdr\\
&+\int_0^t\!\!\!\int_{D}\overline{f}\cdot\varphi
dxdr+\int_0^t\!\!\!\int_{D}\Phi(\overline{u}^N)\cdot \varphi
dxd\overline{W}(r),
\end{split}
\end{equation}
for all $\varphi\in C^{\infty}_{0,\rm{div}}(D)$. It's worth noting
that the limits in the stochastic term is gained by \eqref{2.4} and
\eqref{4.21}
\begin{align*}
\overline{W}^N\rightarrow \overline{W}&  ~\mbox{ in }~
C([0,T],{U}_0),\\
 \Phi(\overline{u}^N)\rightarrow \Phi(\overline{u})&
~\mbox{ in }~ L^2(0,T;L_2({U},L^2(D))
\end{align*}
in probability. By  using Lemma 2.1 in \cite{ADNGRT}, we have
$$\int_0^t\Phi(\overline{u}^N)d\overline{W}^N(s)\rightarrow\int_0^t\Phi(\overline{u})d\overline{W}(s)\ \ \mbox{in}\ \  L^2(0,T;L^2(D)), $$
in probability. Finally, we prove
\begin{equation}\label{4.37}
\begin{split}
\tilde{\overline{S}}=S(x,t,\mathbb{D}(\overline{u})).
\end{split}
\end{equation}
It follows from equation \eqref{4.36}, $\int_{D}\overline{u}\otimes
\overline{u}: \mathbb{D}(\overline{u})dx=0$ and It\^{o}'s formula
that
\begin{equation*}
\begin{split}
\frac{1}{2}\|\overline{u}(t)\|^2_{L^2(D)}
&=\frac{1}{2}\|\overline{u}_0\|^2_{L^2(D)}-\int_0^t\!\!\!\int_D\tilde{\overline{S}}:\mathbb{D}(\overline{u})dxdr-\int_0^t\!\!\!\int_D
\mathcal{S}(\overline{u})\cdot \overline{u}dxdr\\
&\quad+\int_0^t\!\!\!\int_D\overline{f}\cdot
\overline{u}dxdr+\int_0^t\!\!\!\int_D
\overline{u}\cdot \Phi(\overline{u}) dxd\overline{W}(r)\\
&\quad+{\frac{1}{2}\sum_{i=1}^N\int_0^t\!\!\!\int_D|\Phi(\overline{u})e_i|^2dxdr}.
\end{split}
\end{equation*} Similarly,
\begin{equation*}
\begin{split}
\frac{1}{2}\|\overline{u}^N(t)\|^2_{L^2(D)}
&=\frac{1}{2}\|\mathcal{P}^N\overline{u}_0\|^2_{L^2(D)}-\int_0^t\!\!\!\int_DS(x,r,\mathbb{D}(\overline{u}^N)):\mathbb{D}(\overline{u}^N)dxdr\\
&\quad-\int_0^t\!\!\!\int_D \mathcal{S}(\overline{u}^N)\cdot
\overline{u}^Ndxdr+\int_0^t\!\!\!\int_D\overline{f}\cdot
\overline{u}^Ndxdr\\
&\quad+\int_0^t\!\!\!\int_D \overline{u}^N\cdot
\Phi(\overline{u}^N)dxd\overline{W}^{N,N}(r)
+{\frac{1}{2}\sum_{i=1}^N\int_0^t\!\!\!\int_D|\Phi(\overline{u}^N)e_i|^2dxdr}.
\end{split}
\end{equation*}
Subtracting these two equality and applying expectation,
we get
\begin{align*}
&\overline{E}\left(\int_0^T\!\!\!\int_D\!\left(S(x,r,\mathbb{D}(\overline{u}^N))-S(x,r,\mathbb{D}(\overline{u}))\right):
\mathbb{D}(\overline{u}^N-\overline{u})dxdr\right)\\
&\quad+\overline{E}\left(\int_0^T\!\!\!\int_D\left(\mathcal{S}(\overline{u}^N)-\mathcal{S}(\overline{u})\right):\mathbb{D}(\overline{u}^N-\overline{u})dxdr\right)\\
&=\frac{1}{2}\overline{E}\left(\int_D\left(|\overline{u}(T)|^2-|\overline{u}^N(T)|^2\right)dx+
\int_D\left(|\mathcal{P}^N\overline{u}_0^N|^2-|\overline{u}_0|^2\right)dx\right)\\
&\quad+\overline{E}\left(\int_0^T\!\!\!\int_D\left(\tilde{\overline{S}}-S(x,r,\mathbb{D}(\overline{u}^N))\right):\mathbb{D}(\overline{u})dxdr
-\int_0^T\!\!\!\int_DS(x,r,\mathbb{D}(\overline{u})):\mathbb{D}(\overline{u}^N-\overline{u})dxdr\right)\\
&\quad+\overline{E}\left(\int_0^T\!\!\!\int_D\left(\mathcal{S}(\overline{u})-\mathcal{S}(\overline{u}^N)\right):\overline{u}dxdr
-\int_0^T\!\!\!\int_D\mathcal{S}(\overline{u})\cdot(\overline{u}^N-\overline{u})dxdr\right)\\
&\quad+\overline{E}\left(\int_0^T\!\!\!\int_D\overline{f}\cdot(\overline{u}^N-\overline{u})dxdr\right)
+\overline{E}\left({\frac{1}{2}\sum_{i=1}^N\int_0^T\!\!\!\int_D|\Phi(\overline{u}^N)e_i|^2dxdr}\right)\\
&\quad-\overline{E}\left({\frac{1}{2}\sum_{i=1}^N\int_0^T\!\!\!\int_D|\Phi(\overline{u})e_i|^2dxdr}\right).
\end{align*}
By using \eqref{1.4}, \eqref{4.29} and
$\lim\inf_{N\rightarrow\infty}\overline{E}\left[\int_{D}\left(|\overline{u}^N(T)|^2-|\overline{u}(T)|^2\right)dx\right]\geq0$
which follows from the lower semi-continuity and weak convergence of
$\overline{u}^N(T)$, we can infer that
\begin{equation*}
\begin{split}
&\lim_{N\rightarrow\infty}\overline{E}\left[\int_0^T\!\!\!\int_D\left(S(x,r,\mathbb{D}(\overline{u}^N)-S(x,r,\mathbb{D}(\overline{u}))\right)
:\mathbb{D}(\overline{u}^N-\overline{u})dxdr\right]\\
&\leq\frac{1}{2}\lim_{N\rightarrow\infty}\overline{E}\left[\sum_{i=1}^N\int_0^T\!\!\!\int_D(|\Phi(\overline{u}^N)e_i|^2-|\Phi(\overline{u})e_i|^2)dxdr\right].
\end{split}
\end{equation*}
By \eqref{4.20}, \eqref{4.21}, \eqref{2.1} and \eqref{2.2}, we have
\begin{equation*}
\begin{split}
\overline{E}\left({\sum_{i=1}^N\int_0^T\!\!\!\int_D|\Phi(\overline{u}^N)e_i|^2dxdr}\right)
\rightarrow\overline{E}\left({\sum_{i=1}^N\int_0^T\!\!\!\int_D|\Phi(\overline{u})e_i|^2dxdr}\right),
\end{split}
\end{equation*}
after letting $N\rightarrow\infty$. Then
\begin{equation*}
\begin{split}
\lim_{N\rightarrow\infty}\overline{E}\left[\int_0^T\!\!\!\int_D\left(S(x,r,\mathbb{D}(\overline{u}^N)-S(x,r,\mathbb{D}(\overline{u}))\right):
\mathbb{D}(\overline{u}^N-\overline{u})dxdr\right]=0.
\end{split}
\end{equation*}
Thanks to \eqref{1.4} and the monotonicity of $S$, we have
$$\mathbb{D}(\overline{u}^N)\rightarrow \mathbb{D}(\overline{u})\quad \overline{\mathbb P}\otimes\mathcal{L}^{n+1}-\rm{a.e.}.$$
This implies \eqref{4.37} and we complete the proof of Lemma 4.2.
\end{proof}
\begin{Corollary}
Let the assumptions of Lemma 4.2 be satisfied and
$$\int_{L^2_{\rm{div}}(D)}\|v\|^\beta_{L^2(D)}d\mu_0(v)<\infty,\quad \int_{L^2(Q)}\|\mathbf{g}\|^\beta_{L^2(Q)}d\mu_f(\mathbf{g})<\infty$$
for some $\beta\geq2$. Then there exists a martingale weak solution to \eqref{4.1} such that
\begin{equation*}
\begin{split}
E&\left(\sup_{t\in(0,T)}\int_{D}|\overline{u}(t)|^2dx+\int_Q|\nabla \overline{u}|^qdxdt+\varepsilon\int_Q|\overline{u}^N|^{\tilde{q}}dxdt\right)^{\beta/2}\\
\leq
&cE\left(1+\int_{L^2_{\text{div}}(D)}\|v\|^2_{L^2(D)}d\mu_0(v)+\int_{L^2(Q)}\|\mathbf{g}\|^2_{L^2(Q)}d\mu_f(\mathbf{g})\right)^{\beta/2},
\end{split}
\end{equation*}
where $c$ is independent of $\varepsilon$.
\end{Corollary}
\begin{proof}
It follows from \eqref{4.5} that
\begin{equation*}
\begin{split}
&\frac{1}{2}E\left(\sup_{t\in(0,T)}\int_{D}|u^N(t)|^2dx\right)^{\beta/2}\!\!\!+E\left(\int_Q|\nabla u^N|^qdxdt+\varepsilon\int_Q|u^N|^{\tilde{q}}dxdt\right)^{\beta/2}\\
&\lesssim E\!\left(1\!+\!\int_{D}|u_0|^2dx+\int_0^T\!\!\!\!\int_D|f||u^N|dxdr\right)^{\beta/2}\!\!\!+\!\!E\!
\left(\sup_{t\in(0,T)}\left|\int_0^t\!\!\!\int_{D}u^N\cdot \Phi(u^N)dxdW^N(r)\right|\right)^{\beta/2}\\
&\quad+E\!\left(\sum_{i=1}^N\int_0^T\!\!\!\int_D|\Phi(u^Ne_i|^2dxdr\right)^{\beta/2}.
\end{split}
\end{equation*}
In view of Young's inequality, we obtain
\begin{equation*}
\begin{split}
E\left(\int_0^T\!\!\!\int_D|f||u^N|dxdr\right)^{\beta/2}&\leq c(\delta) E\left(\int_Q|f|^2dxdt\right)^{\beta/2}+
\delta E\left[\int_0^T\left(\int_D|u^N|^2dx\right)^{\beta/2}dr\right]\\
&\leq c(\delta) E\left(\int_Q|f|^2dxdt\right)^{\beta/2}+\delta
E\left(\sup_{t\in(0,T)}\int_D|u^N|^2dx\right)^{\beta/2}.
\end{split}
\end{equation*}
By the Burkholder-Davis-Gundy inequality, \eqref{2.1}, H\"{o}lder's
inequality and Young's inequality, one deduces that
\begin{equation*}
\begin{split}
&E\left(\sup_{t\in(0,T)}\left|\int_0^t\!\!\!\int_{D}u^N\cdot \Phi(u^N)dxdW^N(r)\right|\right)^{\beta/2}\\
&=E\left(\sup_{t\in(0,T)}\left|\sum_{i}\int_0^t\!\!\!\int_{D}u^N\cdot g_i(u^N)dxd\beta_i(r)\right|\right)^{\beta/2}\\
&\leq cE\left(\int_0^T\sum_{i}\left(\int_{D}u^N\cdot g_i(u^N)dx\right)^2dt\right)^{\beta/4}\\
&\leq cE\left(\int_0^T\left(\sum_{i}^N\int_{D}|u^N|^2dx\cdot \int_{D}|g_i(u^N)|^2dx\right)dt\right)^{\beta/4}\\
&\leq c(\delta)
E\left(1+\int_0^T\int_{D}|u^N|^2dxdt\right)^{\beta/2}+\delta
E\left(\sup_{t\in(0,T)}\int_D|u^N|^2dx\right)^{\beta/2}.
\end{split}
\end{equation*}
So we have
\begin{equation*}
\begin{split}
E&\left(\sup_{t\in(0,T)}\left(\int_{D}|u^N|^2dx\right)^{\beta/2}\right)+E\left(\int_Q|\nabla u|^qdxdt+\varepsilon\int_Q|u^N|^{\tilde{q}}dxdt\right)^{\beta/2}\\
&\leq
cE\left(1+\int_{D}|u_0|^2dx+\int_Q|f|^2dxdt\right)+cE\left(\int_0^T\left(\int_{D}|u^N|^2dx\right)^{\beta/2}dt\right).
\end{split}
\end{equation*}
We apply Gronwall's inequality to get
\begin{equation}\label{4.38}
\begin{split}
E&\left(\sup_{t\in(0,T)}\int_{D}|u^N|^2dx\right)^{\beta/2}+E\left(\int_Q|\nabla u^N|^qdxdt+\varepsilon\int_Q|u^N|^{\tilde{q}}dxdt\right)^{\beta/2}\\
&\leq cE\left(1+\int_{D}|u_0|^2dx+\int_Q|f|^2dxdt\right)^{\beta/2},
\end{split}
\end{equation}
which gives the claimed inequality.
\end{proof}

\section{Non-stationary Flows}
In this section, we approximate the original equation by some
equations satisfying the assumptions in Section 4. By using Lemma
4.1, we get a solution to this approximated system, meanwhile we get
a priori estimates and a weak convergent subsequence. Finally, we
use the $L^{\infty}$-truncation to pass to the limit in the
nonlinear stress deviator.

\subsection{A priori estimates and weak convergence}

Let's consider the equation:
\begin{equation}\label{5.1}
\begin{cases}
du+\nabla\cdot(u\otimes u-S+p\text{I})dt+\frac{1}{m}|u|^{\tilde{q}-2}udt=fdt+\Phi(u)dW,\\
u|_{t=0}=u_0.
\end{cases}
\end{equation}
From Lemma 4.1 and Lemma 4.2 for $\varepsilon=\frac{1}{m}$, it
follows that there exists a martingale weak solution
$((\Omega,\mathscr{F},(\mathscr{F})_{t\geq0},\mathbb{P}),u^m,u^m_0,f^m,W)$
to \eqref{5.1} with $u^m\in V_{q,\tilde{q}}$, $\mu_0=\mathbb{P}\circ
(u^m_0)^{-1}$ and $\mu_f=\mathbb{P}\circ (f^m)^{-1}$. For
simplicity, we omit the overline. Then, there holds
\begin{equation*}
\begin{split}
&\int_{D}(u^m(t)-u^m_0)\cdot\varphi dx+\frac{1}{m}\int_0^t\!\!\!\int_{D}|u^m|^{\tilde{q}-2}u^mdxd\tau+\int_0^t\!\!\!\int_{D}S(x,r,\mathbb{D}(u^m)):\mathbb{D}(\varphi) dxdr\\
&=\int_0^t\!\!\!\int_{D}u^m\otimes u^m :\mathbb{D}(\varphi)
dxdr+\int_0^t\!\!\!\int_{D}f^m\cdot \varphi dr
dx+\int_0^t\!\!\!\int_{D}\Phi(u^m)\cdot \varphi dx dW(r),
\end{split}
\end{equation*}
for all $\varphi\in C_{0,\text{div}}^{\infty}(D)$.

From \cite{NISW} (beginning of the proof of Thm 2.7 on p.9) we know
that the probability space and the Brownian motion $W$ can be chosen
independently of m. By using Lemma 4.1, we obtain the uniform
estimates for $u^m$:  $$u^m\in L^2(\Omega,;L^\infty(0,T;L^2(D)))\cap
L^{q}(\Omega;L^q(0,T;W_{0,\text{div}}^{1,q}(D))).$$ It follows from
Corollary 4.1 and \eqref{2.4} that
\begin{equation}\label{5.2}
\begin{split}
E\left(\sup_{t\in(0,T)}\left(\int_{D}|u^m|^2dx\right)^{\beta/2}\right)+E\left(\int_Q|\nabla
u^m|^qdxdt+\frac{1}{m}\int_Q|u^m|^{\tilde{q}}dxdt\right)^{\beta/2}
\leq c(\beta).
\end{split}
\end{equation}
With a parabolic interpolation and the choice of $\beta$, we have
\begin{equation}\label{5.3}
\begin{split}
E\left(\sup_{t\in(0,T)}\int_{Q}|u^m|^{r_0}dxdt\right)\leq c,
\end{split}
\end{equation}
for all $r_0:=q\frac{n+2}{n}$, uniformly in $m$. By using
\eqref{5.2}, \eqref{5.3} and the assumption $q>\frac{2n+2}{n+2}$, we
obtain
\begin{equation}\label{5.4}
\begin{split}
E\left(\int_{Q}|u^m\otimes u^m|^{q_0}dxdt+\int_{Q}|\nabla(u^m\otimes
u^m)|^{q_0}dxdt\right)\leq c,
\end{split}
\end{equation}
for some $q_0>1$. After passing to subsequence, one has 
\begin{align}\label{5.5}
u^m\rightharpoonup u &~\mbox{ in }~  L^{\frac{\beta}{2}q} (\Omega
;L^q(0,T;W_{0,\text{div}}^{1,q}(D))),\\
\label{5.6}
u^m\rightharpoonup u& ~\mbox{ in }~  L^{\beta}(\Omega ;L^{\gamma}(0,T;L^2(D))),\; \forall \gamma<\infty,\\
\label{5.7} \frac{1}{m}|u^m|^{\tilde{q}-2}u^m\rightharpoonup 0&
~\mbox{ in }~
L^{\frac{\beta}{2}\tilde{q}'}(\Omega ;L^{\tilde{q}'}(Q)),\\
\label{5.8}
S(x,t,\mathbb{D}(u^m))\rightharpoonup \tilde{S}&  ~\mbox{ in }~  L^{q'}(\Omega;L^{q'}(Q)),\\
\label{5.9}
S(x,t,\mathbb{D}(u^m))\rightharpoonup \tilde{S}& ~\mbox{ in }~   L^{q'} (\Omega;L^{q'}(0,T;W^{-1,q'}(D))),\\
\label{5.10}
u^m\otimes u^m\rightharpoonup U&  ~\mbox{ in }~   L^{q_0}(\Omega;L^{q_0}(0,T;W^{1,q_0}(D))),\\
\label{5.11} \Phi(u^m)\rightharpoonup \tilde{\Phi}&  ~\mbox{ in }~
L^\beta (\Omega;L^\gamma(0,T;L_2(U,L^2(D)))),\; \forall
\gamma<\infty.
\end{align}
Moreover, we have
$$u\in L^{\beta}(\Omega;L^{\infty}(0,T;L^2(D))),$$
$$\tilde{\Phi}\in L^\beta (\Omega;L^\infty(0,T;L_2({U},L^2(D)))).$$
Let
\begin{align*}
H_1^m&:=S(x,t,\mathbb{D}(u^m)),\\
H_2^m&:=\nabla\Delta^{-1}f^m+\nabla\Delta^{-1}\left(\frac{1}{m}|u^m|^{\tilde{q}-2}u^m\right)+u^m\otimes
u^m,\\
\Phi^m&:=\Phi(u^m).
\end{align*}
{From Theorem 3.1 and
Corollary 3.1, we know that there exist the functions} $p_h^m$,
$p_1^m$, $p_2^m$ which are adapted to $\overline{\mathscr{F}}_t$ and
$\Phi_p^m$ which is progressively measurable such that
\begin{equation}\label{5.12}
\begin{split}
&\int_{D}(u^m(t)-u^m_0-\nabla p^m_h(t))\cdot\varphi dx+\int_0^t\!\!\!\int_{D}(H_1^m-p_1^m\text{I}):\nabla\varphi dxdr\\
&=\int_0^t\!\!\!\int_{D}\text{div}(H_2^m-p_2\text{I})\cdot\varphi
dxdr+\int_0^t\!\!\!\int_{D}\Phi^m\cdot\varphi dxdW(r)
+\int_0^t\!\!\!\int_{D}\Phi_p^m \cdot\varphi dxdW(r).
\end{split}
\end{equation}
Using the continuity of $\nabla\Delta^{-1}$ from $L^{q_0}(D)$ to
$W^{1,q_0}(D)$, we have
\begin{align}\label{5.13}
&H_1^m\in L^{\frac{\beta}{2}q'}(\Omega;L^{q'}(Q)),\\
\label{5.14}
&H_2^m\in L^{q_0}(\Omega;L^{q_0}(0,T;W^{1,q_0}(D))),\\
\label{5.15}
&\Phi^m \in  L^\beta (\Omega;L^\infty(0,T;L_2({U},L^2(D)))),
\end{align}
uniformly in $m$. Thanks to the estimates of Theorem 3.1 and
Corollary 3.1, we obtain the following uniform bounds for the
pressure functions:
\begin{align}\label{5.16}
&p_h^m \in  L^\beta (\Omega;L^\infty(0,T;L^2(D))),\\
\label{5.17}
&p_1^m\in L^{\frac{\beta}{2}q'}(\Omega;L^{q'}(Q)),\\
\label{5.18}
&p_2^m\in L^{q_0}(\Omega;L^{q_0}(0,T;W^{1,q_0}(D))),\\
\label{5.19}
&\Phi_p^m \in  L^\beta (\Omega;L^\infty(0,T;L_2({U},L^2(D)))),
\end{align}
uniformly in $m$.

For the pressure function $p_h^m$, since $\Delta p_h^m=0$, by using
regularity theory for harmonic functions and theorem 3.1, one has
\begin{equation}\label{5.20}
\begin{split}
p_h^m\in L^{\beta}(\Omega;L^{\gamma}(0,T;W^{k,\infty}_{loc}(D))),
\end{split}
\end{equation}
for all $k\in\mathbb{N}$. Therefore, for arbitrary $\gamma<\infty$,
we obtain the following convergence:
\begin{align}\label{5.21}
p_h^m\rightharpoonup p_h&~\mbox{ in }~ L^{\beta}(\Omega;L^{\gamma}(0,T;W^{k,\gamma}_{loc}(D))),\\
\label{5.22}
p_1^m\rightharpoonup p_1&~\mbox{ in }~ L^{\frac{\beta}{2}q'}(\Omega;L^{q'}(Q)),\\
\label{5.23}
p_2^m\rightharpoonup p_2&~\mbox{ in }~ L^{q_0}(\Omega;L^{q_0}(0,T;W^{1,q_0}(D))),\\
\label{5.24} \Phi_p^m\rightharpoonup \Phi_p &~\mbox{ in }~ L^\beta
(\Omega;L^\gamma(0,T;L_2({U},L^2(D)))),
\end{align}
after passing to subsequences.

\subsection{Approximate to $u\otimes u$ and $\Phi(u)$}
In this subsection, we  show that the limit functions in \eqref{5.5}
satisfy $U=u\otimes u$ and $\tilde{\Phi}=\Phi(u)$ by using the
tightness of $u^m$. 
It follows from \eqref{5.1}-\eqref{5.3} that
$$E\left(\left\|u^m(t)-\int_0^t\Phi(u^m)dW(r)\right\|_{W^{1,q_0}(0,T;W^{-1,q_0}_{\text{div}}(D))}\right)\leq c.$$
We can deal with the stochastic term similar to \eqref{4.17}. By
using \eqref{5.2} with $r_0>2$ and \eqref{2.1}, we have
$$E\left(\left\|\int_0^t\Phi(u^m)dW(r)\right\|_{C^{\Lambda}([0,T];L^2(D))}\right)\leq c\left(1+\int_{\Omega\times Q}|u^m|^{r_0}dxdtd\mathbb{P}\right)\leq c,$$
for $\Lambda\in[0,1/2)$. Combining the both inequality above, we
obtain
\begin{equation}\label{5.25}
\begin{split}
E\left(\left\|u^m\right\|_{C^{\Lambda}([0,T];W^{-1,q_0}_{\text{div}}(D))}\right)\leq
c,
\end{split}
\end{equation}
and also for some $\lambda>0$
\begin{equation}\label{5.26}
\begin{split}
E\left(\left\|u^m\right\|_{W^{\lambda,q_0}(0,T;W^{-1,q_0}_{\text{div}}(D))}\right)\leq
c.
\end{split}
\end{equation}
On account of \eqref{5.2}, an interpolation
with $L^{q_0}(0,T;W^{1,q_0}_{0,\text{div}}(D))$ shows
\begin{equation}\label{5.27}
\begin{split}
E\left(\left\|u^m\right\|_{W^{\kappa,q_0}(0,T;L^{q_0}_{\text{div}}(D))}\right)\leq
c.
\end{split}
\end{equation}
for some $\kappa>0$.

Next, we prepare the setup for our compactness method. We define the
path space of $(u^m,p_h^m,p_1^m,p_2^m,\Phi_p^m,W,u_0,f)$ by
\begin{equation*}
\begin{split}
\mathcal{V}:&=L^\gamma(0,T;L^\gamma_{\text{div}}(D))\times L^\gamma(0,T;L^\gamma_{loc}(D))\times(L^{q'}(Q),w)\times (L^{q_0}(0,T;W^{1,q_0}(D)),w)\\
&\quad\times (L^\gamma(0,T;L_2({U},L^2(D))),w)\times
C([0,T],{U}_0)\times L^2(D)\times L^2(Q),
\end{split}
\end{equation*}
where $w$ refers to the weak topology. Let us denote by $\nu_{u^m}$,
$\nu_{p_h^m}$, $\nu_{p_1^m}$, $\nu_{p_2^m}$, $\nu_{\Phi_p^m}$,
respectively,  the law of $u^m$, $p_h^m$, $p_1^m$, $p_2^m$ and
$\Phi_p^m$. By $\nu_{W}$, we denote the law of $W$ on
$C([0,T],{U}_0)$. The joint law of $u^m$, $p_h^m$, $p_1^m$, $p_2^m$,
$\Phi_p^m$, $W$, $u_0$ and $f$ on $\mathcal{V}$ is denoted by
$\nu^m$.

\begin{Proposition}
The set $\{\nu^m|m\in\mathbb{N}\}$ is tight on $\mathcal{V}$.
\end{Proposition}
\begin{proof}
In order to prove the tightness of $\nu^m$, we need  the following
five steps.

{Step 1: Tightness of $\nu_{u^m}$.} Since $q>\frac{2n+2}{n+2}$, by
using Remark 1.2 in \cite{JW} and Theorem 5.2 in \cite{HA}, we have
$$W^{\kappa,q_0}(0,T;L^{q_0}_{\text{div}}(D))\cap L^\infty(0,T;L^2(D))\cap L^{q}(0,T;W^{1,q}_{0,\text{div}}(D)))\hookrightarrow L^\gamma(0,T;L^\gamma_{\text{div}}(D))$$
compactly for all $\gamma<q\frac{n+2}{n}$. Choosing a ball $B_R$ in
the space $W^{\kappa,q_0}(0,T;L^{q_0}_{\text{div}}(D))\cap
L^\infty(0,T;L^2(D))\cap L^{q}(0,T;W^{1,q}_{\text{div}}(D))$ and
using \eqref{5.3} and \eqref{5.27}, we obtain
\begin{equation*}
\begin{split}
\nu_{u^m}(B_R^c)&=\mathbb{P}(\|u^m\|_{W^{\kappa,q_0}(0,T;L_{\text{div}}^{q_0}(D))}+\|u^m\|_{L^{q}(0,T;W^{1,q}_{\text{div}}(D))}+\|u^m\|_{L^{\infty}(0,T;L^{2}(D))}\geq R)\\
&\leq
\frac{1}{R}E\left[\|u^m\|_{W^{\kappa,q_0}(0,T;L_{\text{div}}^{q_0}(D))}+\|u^m\|_{L^{q}(0,T;W^{1,q}_{\text{div}}(D))}+\|u^m\|_{L^{\infty}(0,T;L^{2}(D))}\right]\leq\frac{c}{R},
\end{split}
\end{equation*}
where $B_R^c$ is the complement of $B_R$. Then we can find $R(\eta)$
such that
$$\nu_{u^m}(B_{R(\eta)})\geq 1-\frac{\eta}{8},$$
for a fixed $\eta>0$. These imply the tightness of $\nu_{u^m}$.

{Step 2: Tightness of $\nu_{p_h^m}$.} It follows from local
regularity theory for harmonic function and Lebesgue dominate
convergence Theorem (cf. \cite{JW}) that
$$L^\infty(0,T;L^2(D))\cap\{\Delta v(t)=0\, \text{ for \,a.e.}\, t\}\hookrightarrow L^\gamma(0,T;L^\gamma_{loc}(D))$$
is compact for the harmonic pressure $p_h^m$.  We choose a ball
$B_R$ in the space $L^\infty(0,T;L^2(D))\cap\{\Delta v(t)=0\, \text{
for \,a.e.}\, t\}$ and use \eqref{5.16} to obtain
\begin{equation*}
\begin{split}
\nu_{p_h^m}(B_R^c)&=\mathbb{P}(\|p_h^m\|_{L^{\infty}(0,T;L^{2}(D))}\geq R)\leq
\frac{1}{R}E\left[\|p_h^m\|_{L^{\infty}(0,T;L^{2}(D))}\right]\leq\frac{c}{R},
\end{split}
\end{equation*}
where $B_R^c$ is the complement of $B_R$. Hence, we can find
$R(\eta)$ such that
$$\nu_{p_h^m}(B_{R(\eta)})\geq 1-\frac{\eta}{8},$$
for a fixed $\eta>0$. This yield
that the law of $p_h^m$ is also tight.

{Step 3: Tightness of $\nu_{p_1^m}$, $\nu_{p_2^m}$ and
$\nu_{\Phi_p^m}$.} Since the reflexivity of the corresponding
spaces, choosing balls $B_{R_1}$ in the space $L^{q'}(Q)$, $B_{R_2}$
in the space $L^{q_0}([0,T];W^{1,q_0}(D))$, $B_{R}$ in the space
$L^{\infty}([0,T];L_2(U,L^2(D)))$, respectively, and by using
\eqref{5.17}-\eqref{5.19}, we have
\begin{equation*}
\begin{split}
\nu_{p_1^m}(B_{R_1}^c)&=\mathbb{P}(\|p_1^m\|_{L^{q'}(Q)}\geq R_1)\leq
\frac{1}{R_1}E\left[\|p_1^m\|_{L^{q'}(Q)}\right]\leq\frac{c}{R_1},
\end{split}
\end{equation*}
\begin{equation*}
\begin{split}
\nu_{\Phi_p^m}(B_{R_2}^c)&=\mathbb{P}(\|\Phi_p^m\|_{L^{q_0}([0,T];W^{1,q_0}(D))}\geq R_2)\leq
\frac{1}{R_2}E\left[\|\Phi_p^m\|_{L^{q_0}([0,T];W^{1,q_0}(D))}\right]\leq\frac{c}{R_2},
\end{split}
\end{equation*}
\begin{equation*}
\begin{split}
\nu_{p_1^m}(B_R^c)&=\mathbb{P}(\|p_1^m\|_{L^{\infty}([0,T];L_2(U,L^2(D))))}\geq R)\leq
\frac{1}{R}E\left[\|p_1^m\|_{L^{\infty}([0,T];L_2(U,L^2(D)))}\right]\leq\frac{c}{R}.
\end{split}
\end{equation*}
Then we can find compact sets for $p_1^m$, $p_2^m$ and $\Phi_p^m$
with measures greater than $1-\frac{\eta}{8}$(or equal).

{Step 4: Tightness of $\nu_W$.} The law $\nu_W$ is tight as it
coincides with the law of $W$ which is a Radon measure on the Polish
space $C([0,T],{U}_0)$. Then there exists a compact subset
$C_{\eta}\subset C([0,T],{U}_0)$ such that $\nu_{W^m}(C_\eta)\geq
1-\frac{\eta}{8}$.

{Step 5: Tightness of  $\mu_0$ and $\mu_f$.} By the same argument,
we can find compact subsets of $L^2_{\text{div}}(D)$ and $L^2(Q)$
such that $\mu_0$ and $\mu_f$ are smaller than $1-\frac{\eta}{8}$.

So, we can find a compact subset
$\mathcal{V}_\eta\subset\mathcal{V}$ such that
$\nu^m(\mathcal{V}_\eta)\geq1-\eta$. Hence,
$\{\nu^m,m\in\mathbb{N}\}$ is tight in the same space.
\end{proof}
By using the Jakubowski-Skorohod Theorem in \cite{AJ}, we obtain the
following result.
\begin{Proposition}
There exists a probability space
$(\overline{\Omega},\overline{\mathscr{F}},\overline{\mathbb{P}})$
with $\mathcal{V}$-valued Borel measurable random variables
$(\overline{u}^m,\overline{p}_h^m,\overline{p}_1^m,\overline{p}_2^m,\overline{\Phi}_p^m,\overline{W}^m,\overline{u}^m_0,\overline{f}^m)$
and
$(\overline{u},\overline{p}_h,\overline{p}_1,\overline{p}_2,\overline{\Phi}_p,\overline{W},\overline{u}_0,\overline{f})$
such that the following hold:

$(1)$ The laws of
$(\overline{u}^m,\overline{p}_h^m,\overline{p}_1^m,\overline{p}_2^m,\overline{\Phi}_p^m,\overline{W}^m,\overline{u}^m_0,\overline{f}^m)$
and
$(\overline{u},\overline{p}_h,\overline{p}_1,\overline{p}_2,\overline{\Phi}_p,\overline{W},\overline{u}_0,\overline{f})$
under $\overline{\mathbb{P}}$ coincide with $\nu^m$ and
$\nu:=\lim_{m\rightarrow\infty}\nu^m$.

$(2)$ The strong convergence:
\begin{align*}
\overline{u}_0^m\rightarrow \overline{u}_0&~\mbox{ \rm in }~ L^{2}(D)\ \ \overline{\mathbb{P}}-\text{a.s.},\\
\overline{u}^m\rightarrow \overline{u}&~\mbox{ \rm in }~ L^{\gamma}(0,T;L^\gamma(D))\ \ \overline{\mathbb{P}}-\text{a.s.},\\
\overline{p}_h^m\rightarrow \overline{p}_h &~\mbox{ \rm in }~ L^{\gamma}(0,T;L^\gamma_{loc}(D))\ \ \overline{\mathbb{P}}-\text{a.s.},\\
\overline{W}^m\rightarrow \overline{W}& ~\mbox{ \rm in }~ C([0,T],{U}_0)\ \ \overline{\mathbb{P}}-\text{a.s.},\\
\overline{f}^m\rightarrow \overline{f}& ~\mbox{ \rm in }~
L^{2}(0,T;L^2(D))\ \ \overline{\mathbb{P}}-\text{a.s.}.
\end{align*}

$(3)$ The weak convergence:
\begin{align*}\overline{p}_1^m\rightharpoonup \overline{p}_1 &~\mbox{ \rm in }~ L^{q'}(Q)\ \ \overline{\mathbb{P}}-\text{a.s.},\\
\overline{p}_2^m\rightharpoonup \overline{p}_2 &~\mbox{ \rm in }~ L^{q_0}(0,T;W^{1,q_0}(D))\ \ \overline{\mathbb{P}}-\text{a.s.},\\
\overline{\Phi}_p^m\rightharpoonup \overline{\Phi}_p &~\mbox{ \rm in
}~ L^r(0,T;L_2({U},L^2(D)))\ \ \overline{\mathbb{P}}-\text{a.s.}.
\end{align*}

$(4)$
$$\int_{\overline{\Omega}}\left(\sup_{t\in[0,T]}\|\overline{W}^m(t)\|^{\alpha}_{{U}_0}\right)d\overline{\mathbb{P}}
=\int_{\Omega}\left(\sup_{t\in[0,T]}\|W(t)\|^{\alpha}_{{U}_0}\right)d\mathbb{P},$$
for all $\alpha<\infty$.
\end{Proposition}

By virtue of the equality of laws, we obtain the weak convergence:
\begin{align*}\overline{p}_1^m\rightharpoonup \overline{p}_1 &~\mbox{ in }~ L^{q'}(\overline{\Omega};L^{q'}Q)),\\
\overline{p}_2^m\rightharpoonup \overline{p}_2&  ~\mbox{ in }~ L^{q_0}(\overline{\Omega};L^{q_0}(0,T;W^{1,q_0}(D))),\\
\overline{\Phi}_p^m\rightharpoonup \overline{\Phi}_p & ~\mbox{ in }~  L^{q_0}(\overline{\Omega};L^\gamma(0,T;L_2({U},L^2(D)))).\\
\end{align*}

By Vitali's convergence Theorem, we get the strong convergence:
\begin{align}\label{5.28}
\overline{W}^m\rightarrow \overline{W}&~\mbox{ in }~
L^2(\overline{\Omega};C([0,T],{U}_0)),\\
\label{5.29}
\overline{u}^m\rightarrow \overline{u} &~\mbox{ in }~ L^{\gamma}(\overline{\Omega}\times Q;\overline{\mathbb{P}}\otimes\mathcal{L}^{n+1}),\\
\label{5.30} \overline{u}_0^m\rightarrow \overline{u}_0 &~\mbox{ in
}~ L^{2}(\overline{\Omega}\times
D;\overline{\mathbb{P}}\otimes\mathcal{L}^{n+1}),\\
\label{5.31} \nabla^k\overline{p}_h^m\rightarrow
\nabla^k\overline{p}_h &~\mbox{ in }~
L^{\gamma}(\overline{\Omega}\times (0,T)\times
D';\overline{\mathbb{P}}\otimes\mathcal{L}^{n+1}),\\
\label{5.32} \overline{f}^m\rightarrow \overline{f} &~\mbox{ in }~
L^{2}(\overline{\Omega}\times
Q;\overline{\mathbb{P}}\otimes\mathcal{L}^{n+1}),
\end{align}
for all $\gamma<q\frac{n+2}{n}$ and all $D'\subset\subset D$, after
choosing a subsequence. For the harmonic pressure \eqref{5.31},
applying local regularity theory for harmonic maps above, one has for all $s<\infty$
\begin{align}\label{5.33}
\overline{u}^m\otimes\overline{u}^m\rightharpoonup
\overline{u}\otimes\overline{u} &~\mbox{ in }~  L^{q_0} (\Omega;L^{q_0}(0,T;W^{1,q_0}(D))),\\
\label{5.34} \Phi(\overline{u}^m)\rightharpoonup \Phi(\overline{u})
&~\mbox{ in }~
L^\beta (\Omega;L^{{s}}(0,T;L_2({U},L^2(D)))),\\
\label{5.35} \Phi_p(\overline{u}^m)\rightharpoonup
\Phi_p(\overline{u}) &~\mbox{ in }~  L^\beta
(\Omega;L^{{s}}(0,T;L_2({U},L^2(D)))).
\end{align}

Let $\overline{\mathscr{F}}_t$ be the
$\overline{\mathbb{P}}$-augmented canonical filtration of the
process
$(\overline{u},\overline{p}_h,\overline{p}_1,\overline{p}_2,\overline{\Phi}_p,\overline{W},\overline{f})$,
that is,
$$\overline{\mathscr{F}_t}=\sigma(\sigma(\varrho_t\overline{u},\varrho_t\overline{p}_h,\varrho_t\overline{p}_1,\varrho_t\overline{p}_2,\varrho_t\overline{\Phi}_p,
\varrho_t\overline{W},\varrho_t\overline{f})\cup\{\mathcal{N}\in\overline{\mathscr{F}};\overline{\mathbb{P}}(\mathcal{N})=0\}),$$
for $t\in [0,T]$. As done in the proof of Lemma 4.2, we can also
show that the equation hold on the new probability space, i.e.,
\begin{equation*}
\begin{split}
&\int_{D}(\overline{u}^m(t)-\overline{u}^m_0-\nabla \overline{p}^m_h(t))\cdot\varphi dx+\int_0^t\!\!\!\int_{D}(\overline{H}_1^m-\overline{p}_1^m\textrm{I}):\nabla\varphi dxdr\\
&=\int_0^t\!\!\!\int_{D}\text{div}(\overline{H}_2^m-\overline{p}_2\textrm{I})\cdot\varphi
dxdr+\int_0^t\!\!\!\int_{D}\Phi(\overline{u}^m)\cdot\varphi
dxd\overline{W}(r)+\int_0^t\!\!\!\int_{D}\overline{\Phi}_p^m\cdot\varphi
 dxd\overline{W}(r),
\end{split}
\end{equation*}
$\overline{\mathbb{P}}\otimes \mathcal{L}^1$-a.e. for all
$\varphi\in C_0^\infty(D)$, where
\begin{align*}
\overline{H}_1^m&:=S(x,t,\mathbb{D}(\overline{u}^m)),\\
\overline{H}_2^m&:=\overline{u}^m\otimes
\overline{u}^m+\nabla\Delta^{-1}\left(\frac{1}{m}|\overline{u}^m|^{\tilde{q}-2}\overline{u}^m\right)
+\nabla\Delta^{-1}\overline{f}^m.
\end{align*}
\begin{Remark}
Here we use the test-functions $\varphi\in C_0^\infty(D)$, instead
of $\varphi\in C_{0,\text{div}}^\infty(D)$.
\end{Remark}

Using  Lemma 2.1 in \cite{ADNGRT} and the convergence
\eqref{5.28}-\eqref{5.35}, we obtain the limit equation:
\begin{equation}\label{5.36}
\begin{split}
&\int_{D}(\overline{u}(t)-\overline{u}_0-\nabla \overline{p}_h(t))\cdot\varphi dx+\int_0^t\!\!\!\int_{D}(\overline{H}_1-\overline{p}_1\text{I}):\nabla\varphi dxdr\\
&=\int_0^t\!\!\!\int_{D}\text{div}(\overline{H}_2-\overline{p}_2\textrm{I})\cdot\varphi
dxdr+\int_0^t\!\!\!\int_{D}\Phi(\overline{u})\cdot\varphi
dxd\overline{W}(r)+\int_0^t\!\!\!\int_{D}\overline{\Phi}_p\cdot\varphi
dxd\overline{W}(r),
\end{split}
\end{equation}
for all $\varphi\in C_0^\infty(D)$, where
$$\overline{H}_1:=\tilde{\overline{S}},\; \overline{H}_2:=\overline{u}\otimes\overline{u}+\nabla\Delta^{-1}\overline{f}.$$
It remains to show
$\tilde{\overline{S}}=S(x,t,\mathbb{D}(\overline{u}))$. Let
\begin{align*}
\overline{G}_1^m:&=S(x,t,\mathbb{D}(\overline{u}^m))-\tilde{\overline{S}},\\
\overline{G}_2^m:&=\overline{u}^m\otimes\overline{u}^m-\overline{u}\otimes\overline{u}+
\nabla\Delta^{-1}\left(\frac{1}{m}|\overline{u}^m|^{\tilde{q}-2}\overline{u}^m\right)+\nabla\Delta^{-1}(\overline{f}^m-\overline{f}),\\
\overline{\Phi}^m:&=(\Phi(\overline{u}^m),\;-\Phi(\overline{u})), \quad \overline{\Phi}_\vartheta^m:=(\Phi_p(\overline{u}^m),-\Phi_p(\overline{u})),\\
\overline{\vartheta}_h^m:&=\overline{p}_h^m-\overline{p}_h,\;\overline{\vartheta}_1^m:
=\overline{p}_1^m-\overline{p}_1,\;\overline{\vartheta}_2^m:=\overline{p}_2^m-\overline{p}_2.
\end{align*}
Then the following convergence hold:
\begin{align}\label{5.39}
\overline{u}^m-\overline{u}\rightharpoonup 0& ~\mbox{ in }~
L^{\frac{\beta}{2}q} (\overline{\Omega};L^q(0,T;W_{0}^{1,q}(D))),\\
\label{5.40} \overline{u}^m-\overline{u}\rightharpoonup 0& ~\mbox{
in }~
L^{\beta}(\overline{\Omega};L^{\gamma}(0,T;L^2(D))),\; \forall \gamma<\infty,\\
\label{5.41} \overline{G}_1^m\rightharpoonup 0 &~\mbox{ in }~
L^{\frac{\beta}{2}q'}(\overline{\Omega};L^{q'}(Q)),\\
\label{5.42} \overline{G}_2^m\rightharpoonup 0& ~\mbox{ in }~
L^{q_0}
(\overline{\Omega};L^{q_0}(0,T;W^{1,q_0}(D))),\\
\label{5.43} \overline{\Phi}^m-\overline{\Phi}\rightharpoonup 0&
~\mbox{ in }~
L^{\beta}(\overline{\Omega};L^{\gamma}(0,T;L_2({U},L^2(D)))),\;
\forall \gamma<\infty,
\end{align}
where $\overline{\Phi}=(\Phi(\overline{u}),-\Phi(\overline{u})).$
For the pressure functions, we have
\begin{align}\label{5.44}
\overline{\vartheta}_h^m\rightarrow 0& ~\mbox{ in }~
L^{\beta}(\overline{\Omega};L^{\gamma}(0,T;W^{k,\gamma}_{loc}(D))),\;
\forall
\gamma<\infty,\\
\label{5.45} \overline{\vartheta}_1^m\rightharpoonup 0& ~\mbox{ in
}~
L^{\frac{\beta}{2}q'}(\overline{\Omega};L^{q'}(Q)),\\
\label{5.46} \overline{\vartheta}_2^m\rightharpoonup 0& ~\mbox{ in
}~ L^{q_0}
(\overline{\Omega};L^{q_0}(0,T;W^{1,q_0}(D))),\\
\label{5.47}
\overline{\Phi}_{\vartheta}^m-\overline{\Phi}_\vartheta\rightharpoonup
0 &~\mbox{ in }~
L^{\beta}(\overline{\Omega};L^{\gamma}(0,T;L_2({U},L^2(D)))),\;
\forall \gamma<\infty.
\end{align}
Moreover, we obtain
\begin{align}\label{5.48}
&\overline{\vartheta}_h^m \in L^{\beta}(\overline{\Omega};L^{\infty}(0,T;L^2(D))),\\
\label{5.49}
&\overline{\Phi}^m\in L^{\beta}(\overline{\Omega};L^{\infty}(0,T;L_2({U},L^2(D)))),\\
\label{5.50} &\overline{\Phi}_{\vartheta}^m\in
L^{\beta}(\overline{\Omega};L^{\infty}(0,T;L_2({U},L^2(D)))),
\end{align}
uniformly in $m$.

The difference of approximates equation and limit equation read as
\begin{align}\label{5.51}
\notag&\int_{D}(\overline{u}^m(t)-\overline{u}(t)+\overline{u}_0-\overline{u}^m_0-\nabla
\overline{\vartheta}^m_h(t))\cdot\varphi dx+
\int_0^t\!\!\!\int_{D}(\overline{G}^m_1-\overline{\vartheta}^m_1\text{I}):\nabla\varphi dxdr\\
&=\int_0^t\!\!\!\int_{D}\text{div}(\overline{G}^m_2-\overline{\vartheta}^m_2\text{I})\cdot\varphi
dxdr+\int_0^t\!\!\!\int_{D}\overline{\Phi}^m\cdot\varphi
dxd(\overline{W}^m(r),\overline{W}(r))\\
\notag&\quad+\int_0^t\!\!\!\int_{D}\overline{\Phi}_{\vartheta}^m\cdot\varphi
dx d(\overline{W}^m(r),\overline{W}(r))
\end{align}
for all $\varphi\in C_0^\infty(D)$. Define
$\overline{v}^m=\overline{u}^m-\nabla\overline{\vartheta}_h^m$ and
denote $\overline{v}^{m,k}:=\overline{v}^m-\overline{v}^k$, $m\geq
k$. Similarly, we define $\overline{G}_1^{m,k}$,
$\overline{G}_2^{m,k}$, $\overline{\vartheta}_1^{m,k}$,
$\overline{\vartheta}_2^{m,k}$, $\overline{\Phi}^{m,k}$ and
$\overline{\Phi}_{\vartheta}^{m,k}$. Then, we have
\begin{align}\label{5.52}
&\overline{v}^m\rightharpoonup 0 ~\mbox{ in }~ L^{q}
(\overline{\Omega};L^{q}(0,T;W^{1,q_0}_0(D))),\\
\label{5.53} &\overline{v}^m\rightarrow 0 ~\mbox{ in }~
L^{\gamma}(\overline{\Omega}\times (0,T)\times
D';\overline{\mathbb{P}}\otimes\mathcal{L}^{n+1}),
\end{align}
and
\begin{equation}\label{5.54}
\begin{split}
&\int_{D}(\overline{v}^{m,k}-\overline{v}^{m,k}_0)\cdot\varphi dx+\int_0^t\!\!\!\int_{D}(\overline{G}^{m,k}_1-\overline{\vartheta}^{m,k}_1\textrm{I}):\nabla\varphi dxdr\\
&=\int_0^t\!\!\!\int_{D}\text{div}(\overline{G}^{m,k}_2-\overline{\vartheta}^{m,k}_2\textrm{I})\cdot\varphi
dxdr+\int_0^t\!\!\!\int_{D}\overline{\Phi}^{m,k}\cdot\varphi
dxd(\overline{W}^m(r),\overline{W}^k(r))\\
&\quad+\int_0^t\!\!\!\int_{D}\overline{\Phi}_{\vartheta}^{m,k}\cdot\varphi
dx d(\overline{W}^m(r),\overline{W}^k(r)),
\end{split}
\end{equation}
for all $\varphi\in C_0^\infty(D)$.

\subsection{$L^\infty$-truncation}
From density arguments, we are allowed to test the equations with
$\varphi\in W_0^{1,p}\cap L^\infty(D)$. Since the function
$\overline{u}(w,t,\cdot)$ does not belong to this class,  the
$L^\infty$-truncation {is used to} the deterministic
problem in \cite{JW}. In this subsection, we apply the
$L^\infty$-truncation to the stochastic setting.

Let
\begin{align*}h_L(s):=\int_0^s\Psi_L(\theta)\theta d\theta,\;
H_L(\xi):=h_L(|\xi|),\; \Psi_L:=\sum_{l=1}^L\psi_{2^{-l}},\;
\psi_{\delta}:=\psi(\delta s),
\end{align*}
for $L\in\mathbb{N}_0$, where $\psi\in C_0^\infty([0,2])$,
$\psi\equiv0$ on $[0,1]$, $0\leq\psi\leq1$ and $0\leq-\psi'\leq2$.
Denote
$$f_L(u):=\int_D\eta H_L(u)dx,~\mbox{ for }~\eta\in C_0^\infty(D).$$
By using It\^{o}'s formula, we have
\begin{align*}
&\int_{D}\eta H_L(\overline{v}^{m,k}(t))dx\\
&=f_L(\overline{v}^{m,k}(0))+\int_0^tf'_L(\overline{v}^{m,k})d\overline{v}^{m,k}+\frac{1}{2}\int_0^tf''_L(\overline{v}^{m,k})d\langle\overline{v}^{m,k}\rangle(r)\\
&=\int_{D}\eta
H_L(\overline{v}^{m}_0-\overline{v}^{k}_0)dx-\int_0^t\!\!\!\int_{D}\eta(\overline{G}_1^{m,k}-\overline{\vartheta}_1^{m,k}\textrm{I}):\nabla(\Psi_L(|\overline{v}^{m,k}|)
\overline{v}^{m,k})dxdr\\
&\quad-\int_0^t\!\!\!\int_{D}(\overline{G}_1^{m,k}-\overline{\vartheta}_1^{m,k}\textrm{I}):\nabla\eta\otimes(\Psi_L(|\overline{v}^{m,k}|)\overline{v}^{m,k})dxdr\\
&\quad+\int_0^t\!\!\!\int_{D}\eta\Psi_L(|\overline{v}^{m,k}|)\text{div}\left(\overline{G}_2^{m,k}-\overline{\vartheta}_2^{m,k}\textrm{I}\right)\cdot\overline{v}^mdxdr\\
&\quad+\int_{D}\!\!\int_0^t\eta\Psi_L(|\overline{v}^{m,k}|)\overline{v}^{m,k}\cdot \left(\Phi(\overline{u}^{m,k})d\overline{W}^m(r)-\Phi(\overline{u}^{k})d\overline{W}^k(r)\right)dx\\
&\quad+\int_{D}\!\!\int_0^t\eta\Psi_L(|\overline{v}^{m,k}|)\overline{v}^{m,k}\cdot
\left(\Phi_\vartheta(\overline{u}^{m,k})d\overline{W}^m(r)-\Phi_\vartheta(\overline{u}^{k})
d\overline{W}^k(r)\right)dx\\
&\quad+\frac{1}{2}\int_{D}\!\!\int_0^t\eta \mathbb{D}^2H_L(\overline{v}^{m,k})d\langle\int_0^\cdot\Phi(\overline{u}^m)d\overline{W}^m-\int_0^\cdot\Phi(\overline{u}^k)d\overline{W}^k\rangle(r) dx\\
&\quad+\frac{1}{2}\int_{D}\!\!\int_0^t\eta
\mathbb{D}^2H_L(\overline{v}^{m,k})d\langle\int_0^\cdot\Phi_\vartheta(\overline{u}^m)d\overline{W}^m-
\int_0^\cdot\Phi_\vartheta(\overline{u}^k)d\overline{W}^k\rangle(r) dx\\
&=:J_1+J_2+J_3+J_4+J_5+J_6+J_7+J_8.
\end{align*}
$\overline{E}[J_1]\rightarrow0$ if $m,k\rightarrow\infty$, gained by
equation \eqref{5.30} and
$\overline{v}^m(0)-\overline{v}^k(0)=\overline{u}^m(0)-\overline{u}^k(0)$
(see Theorem \ref{theorem3.1} (2)). We are going to show that the
expectation values of $J_3$-$J_7$ vanish if $m,k\rightarrow\infty$.
By using the monotone operator theory, we obtain
$\mathbb{D}(u^m)\rightarrow \mathbb{D}(u)$, a.e.. Clearly,
$\Psi_L(|\overline{v}^{m,k}|)\overline{v}^{m,k}$ are bounded in
$L^\gamma$. By virtue of \eqref{5.29} and the construction of
$\Psi_L$, {after taking a subsequence}, we
have
\begin{align}\label{5.55}
\Psi_L(|\overline{v}^{m,k}|)\overline{v}^{m,k}\rightarrow 0~\mbox{
in }~ L^\gamma(\overline{\Omega}\times
Q;\overline{\mathbb{P}}\otimes\mathcal{L}^{n+1})~\mbox{ as }~
m,k\rightarrow\infty,
\end{align}
for all $\gamma<\infty$. It follows from \eqref{5.39} and
\eqref{5.44} that $\overline{E}[J_3]\rightarrow0$,
$\overline{E}[J_4]\rightarrow0$ if $m,k\rightarrow\infty$. Clearly,
$\overline{E}[J_5]=0$, $\overline{E}[J_6]=0$. Since
$|\mathbb{D}^2H_L|\leq c(L)$, we have
\begin{align*}
J_7&\leq c\sum_{\ell=1}^n\int_{D}\!\!\int_0^td\langle\int_0^\cdot(\Phi(\overline{u}^m)-\Phi(\overline{u}^k))d\overline{W}^m\rangle^{{\ell\ell}}(r) dx\\
&\quad+c\sum_{\ell=1}^n\int_{D}\!\!\int_0^td\langle\int_0^\cdot\Phi(\overline{u}^k)d(\overline{W}^m-\overline{W}^k)\rangle^{\ell\ell}(r) dx\\
&\quad+c\sum_{\ell=1}^n\int_{D}\!\!\int_0^td\langle\int_0^\cdot(\Phi(\overline{u}^m)-\Phi(\overline{u}^k))d\overline{W}^m,
\int_0^\cdot\Phi(\overline{u}^k)d(\overline{W}^m-\overline{W}^k)\rangle^{\ell\ell}(r) dx\\
&\leq c\sum_{\ell=1}^n\int_{D}\!\int_0^td\langle\int_0^\cdot(\Phi(\overline{u}^m)-\Phi(\overline{u}^k))d\overline{W}^m\rangle^{\ell\ell}(r) dx\\
&\quad+c\sum_{\ell=1}^n\int_{D}\!\int_0^td\langle\int_0^\cdot\Phi(\overline{u}^k)d(\overline{W}^m-\overline{W}^k)\rangle^{\ell\ell}(r) dx\\
&=J_{71}+J_{72}.
\end{align*}

By using \eqref{2.1}, \eqref{2.2} and \eqref{5.29}, we obtain
\begin{equation*}
\begin{split}
\overline{E}(J_{71})&\leq c\overline{E}\left(\int_0^t\|\Phi(\overline{u}^m)-\Phi(\overline{u}^k)\|^2_{L_2({U},L^2(D))}dr\right)\\
&\leq
c\overline{E}\left(\int_0^t\!\!\!\int_{D}|\overline{u}^m-\overline{u}^k|^2dxdr\right)\rightarrow0,\quad
m,k\rightarrow\infty
\end{split}
\end{equation*}

In view of  \eqref{2.3}, \eqref{5.28} and $\overline{u}^k\in
L^2(\overline{\Omega}\times
Q;\overline{\mathbb{P}}\otimes\mathcal{L}^{n+1})$ uniformly in $k$,
one deduces that
\begin{align*}
\overline{E}(J_{72})&= \overline{E}\left(\int_0^t\sum_i\left(\int_D|g_i(\overline{u}^k)|^2\text{Var}\left(\overline{\beta}_i^m(1)
-\overline{\beta}_i^k(1)\right)dx\right)dt\right)\\
&\leq c\overline{E}\left(\int_0^t\left(\int_{D}\sup_ii^2|g_i(\overline{u}^k)|^2dx\right)dt\right)\sum_i\frac{1}{i^2}\text{Var}
\left(\overline{\beta}_i^m(1)-\overline{\beta}_i^k(1)\right)\\
&\leq c\overline{E}\left(\int_0^t\int_{D}(1+|\overline{u}^k|^2)dxdt\right)\cdot\overline{E}\left(\|\overline{W}^m-\overline{W}^k\|^2_{C([0,T],{U}_0)}\right)\\
&\rightarrow0,\quad m,k\rightarrow\infty.
\end{align*}

From Corollary 3.1 and the usage of the cut-off function $\eta$, we
know that $\Phi_\vartheta$ inherits the properties of $\Phi$. So we
can estimate $J_8$ by the same method.  Plugging all together, we
have
\begin{align}\label{5.56}
\notag\lim\sup_{m}&\overline{E}\left(\int_Q
\eta(S(x,r,\mathbb{D}(\overline{u}^m))-\tilde{\overline{S}}):
\Psi_L(|\overline{v}^{m}-\overline{v}|)\mathbb{D}(\overline{v}^{m}-\overline{v})dxdr\right)\\
\notag&\leq \lim\sup_{m}\overline{E}\left(\int_Q
\eta(S(x,r,\mathbb{D}(\overline{u}^m))-\tilde{\overline{S}}):\nabla\{\Psi_L(|\overline{v}^{m}-\overline{v}|)\}
\otimes(\overline{v}^{m}-\overline{v}) dx dr\right)\\
&\quad+\lim\sup_{m}\overline{E}\left(\int_Q
\eta\overline{\vartheta}_1^{m}\text{div}(\Psi_L(|\overline{v}^{m}-\overline{v}|)(\overline{v}^{m}-\overline{v}))dxdr\right).
\end{align}
Since $\text{div}(\overline{v}^{m}-\overline{v})=0$, there holds
\begin{equation*}
\begin{split}
&\lim\sup_{m}\overline{E}\left(\int_Q \eta\overline{\vartheta}_1^{m}\text{div}(\Psi_L(|\overline{v}^{m}-\overline{v}|)(\overline{v}^{m}-\overline{v}))dxdr\right)\\
&=\lim\sup_{m}\overline{E}\left(\int_Q
\eta\overline{\vartheta}_1^{m}\nabla\{\Psi_L(|\overline{v}^{m}-\overline{v}|)\}\cdot(\overline{v}^{m}-\overline{v})dxdr\right).
\end{split}
\end{equation*}

Note that, for all $\ell\in \mathbb{N}_0$,
\begin{equation*}
\begin{split}
|\nabla\{\psi_{2^{-\ell}}(|\overline{v}^{m}-\overline{v}|)\}\cdot(\overline{v}^{m}-\overline{v})|&\leq|\psi'_{2^{-\ell}}(|\overline{v}^{m}-\overline{v}|)(\overline{v}^{m}-\overline{v})\otimes\nabla(\overline{v}^{m}-\overline{v})|\\
&\leq-2^{-\ell}|\overline{v}^{m}-\overline{v}|\psi'(2^{-\ell}|\overline{v}^{m}-\overline{v}|)|\nabla(\overline{v}^{m}-\overline{v})|\\
&\leq c|\nabla(\overline{v}^{m}-\overline{v})|_{\chi_{A_\ell}},
\end{split}
\end{equation*}
where $A_\ell:=\{2^\ell<|\overline{v}^{m}-\overline{v}|\leq
2^{\ell+1}\}$. This yields
\begin{equation*}
\begin{split}
|\nabla \Psi_L(|\overline{v}^{m}-\overline{v}|)(\overline{v}^{m}-\overline{v})|&\leq\sum_{\ell=0}^L|\nabla\{\psi_{2^{-\ell}}(|\overline{v}^{m}-\overline{v}|)\}(\overline{v}^{m}-\overline{v})|\\
&\leq
c\sum_{\ell=0}^L|\nabla(\overline{v}^{m}-\overline{v})|_{\chi_{A_\ell}}\leq
c|\nabla(\overline{v}^{m}-\overline{v})|.
\end{split}
\end{equation*}
By using \eqref{5.39} and \eqref{5.44}, we have
\begin{equation}\label{5.59}
\begin{split}
\nabla
\Psi_L(|\overline{v}^{m}-\overline{v}|)(\overline{v}^{m}-\overline{v})\in
L^q(\overline{\Omega}\times
Q;\overline{\mathbb{P}}\otimes\mathcal{L}^{n+1}),
\end{split}
\end{equation}
uniformly in $L$ and $m$. Then, we can  conclude that
\begin{equation}\label{5.60}
\begin{split}
\lim\sup_{m}&\overline{E}\left(\int_Q
\eta(S(x,r,\mathbb{D}(\overline{u}^m))-\tilde{\overline{S}}):
\Psi_L(|\overline{v}^{m}-\overline{v}|)\mathbb{D}(\overline{v}^{m}-\overline{v})dxdr\right)\leq
K.
\end{split}
\end{equation}

In view of \eqref{5.60}, using Cantor's diagonalizing principle,
there exists a subsequence with
$$\sigma_{\ell,m_\ell}\!:=\overline{E}\left(\int_Q \eta(S(x,r,\mathbb{D}(\overline{u}^{m_\ell}))\!-
\!\tilde{\overline{S}}):\psi_{2^{-\ell}}(|\overline{v}^{m_{\ell}}-\overline{v}|)\mathbb{D}(\overline{v}^{m_{\ell}}-\overline{v})dxdr\right)\rightarrow\sigma_\ell,$$
for $\ell\in \mathbb{N}_0$, as $\ell\rightarrow\infty$. From
\eqref{1.4}, we know that $\sigma_\ell\geq0$ for all $\ell\in
\mathbb{N}_0$
 and $\sigma_\ell$ is increasing in $\ell$. Thanks to \eqref{5.60}, we have
$$0\leq\sigma_0\leq\frac{\sigma_0+\sigma_2+\cdots+\sigma_\ell}{\ell}\leq\frac{K}{\ell},$$
for all $\ell\in \mathbb{N}$. Hence $\sigma_0=0$ and therefore
$$\overline{E}\left(\int_Q (S(x,r,\mathbb{D}(\overline{u}^m))\!-\tilde{\overline{S}}):\psi_1(|\overline{v}^{m}-\overline{v}|)
\mathbb{D}(\overline{v}^{m}-\overline{v})dxdr\right)\rightarrow0~\mbox{
as }~ m\rightarrow0.$$ It follows from \eqref{5.31} that
\begin{equation}\label{5.61}
\begin{split}
\overline{E}\left(\int_Q
(S(x,r,\mathbb{D}(\overline{u}^m))\!-\tilde{\overline{S}}):\psi_1(|\overline{v}^{m}-\overline{v}|)\mathbb{D}
(\overline{u}^{m}-\overline{u})dxdr\right)\rightarrow0~\mbox{ as }~
m\rightarrow0.
\end{split}
\end{equation}
Using \eqref{5.41} and the fact
$\psi_{2^-N}(|\overline{v}^m-\overline{v}|)\rightarrow1$ as
$m\rightarrow\infty$, one has
\begin{equation}\label{5.62}
\begin{split}
\lim\sup_{m}\overline{E}\left(\int_Q
S(x,r,\mathbb{D}(\overline{u}^m))\!:\psi_1(|\overline{v}^{m}-\overline{v}|)\mathbb{D}
(\overline{u}^{m})dxdr\right)=\overline{E}\left(\int_Q\tilde{\overline{S}}\!:\mathbb{D}
(\overline{u})dxdr\right).
\end{split}
\end{equation}
Lemma A.2 in \cite{JW} implies that
$\tilde{\overline{S}}=S(x,t,\mathbb{D}(\overline{u}))$. Then we
complete the proof of Theorem 2.1.

\section*{Acknowledgments}
The work of Z. Tan and Y.C. Wang was supported by the National
Natural Science Foundation of China (No. 11271305, 11531010). H.Q.
Wang's research was supported by National Postdoctoral Program for
Innovative Talents (No. BX201600020).

\end{document}